\documentclass[onefignum,onetabnum]{siamonline190516}

\usepackage{lipsum}
\usepackage{amsfonts}
\usepackage{hyperref}
\usepackage{graphicx}
\usepackage{tabularx}
\usepackage{booktabs}
\usepackage{multirow}
\usepackage{footnote}
\usepackage{shuffle}
\usepackage{epstopdf}
\usepackage{algorithmic}
\usepackage{forest}
\usepackage{enumitem}
\usepackage{tikz}
\usetikzlibrary{trees}
\usepackage{tikz-cd}
\usepackage{subcaption}
\usepackage{amssymb}

\newcommand{\norm}[1]{\left\lVert#1\right\rVert}

\ifpdf
  \DeclareGraphicsExtensions{.eps,.pdf,.png,.jpg}
\else
  \DeclareGraphicsExtensions{.eps}
\fi

\usepackage{enumitem}
\setlist[enumerate]{leftmargin=.5in}
\setlist[itemize]{leftmargin=.5in}

\newsiamremark{remark}{Remark}
\newsiamremark{hypothesis}{Hypothesis}
\crefname{hypothesis}{Hypothesis}{Hypotheses}
\newsiamthm{claim}{Claim}

\headers{The Signature Kernel is the solution of a Goursat PDE}{C. Salvi, T. Cass, J. Foster, T. Lyons, W. Yang}

\title{The Signature Kernel is the solution of a Goursat PDE \thanks{Submitted to the editors \today.
\funding{CS was supported by the
EPSRC grant EP/R513295/1. All the authors were supported by DataSig under the ESPRC grant EP/S026347/1 and by the Alan Turing Institute under the EPSRC grant EP/N510129/1.}}}

\author{
Cristopher Salvi\thanks{University of Oxford \& The Alan Turing Institute 
(\email{cristopher.salvi@maths.ox.ac.uk},
\email{james.foster@maths.ox.ac.uk}, \email{terry.lyons@maths.ox.ac.uk}, \email{weixin.yang@maths.ox.ac.uk}).}
\and Thomas Cass\thanks{Imperial College London \& The Alan Turing Institute 
(\email{thomas.cass@imperial.ac.uk}).}
\and James Foster \footnotemark[2]
\and Terry Lyons\footnotemark[2]
\and Weixin Yang\footnotemark[2]
}

\usepackage{amsopn}

\ifpdf
\hypersetup{
  pdftitle={The Signature Kernel is the solution of a Goursat PDE},
  pdfauthor={C. Salvi, T. Cass, J. Foster, T. Lyons, W. Yang}
}
\fi

\begin{document}

\maketitle

\begin{abstract}
    Recently, there has been an increased interest in the development of kernel methods for learning with sequential data. The signature kernel is a learning tool with potential to handle irregularly sampled, multivariate time series. In \cite{kiraly2019kernels} the authors introduced a kernel trick for the truncated version of this kernel avoiding the exponential complexity that would have been involved in a direct computation. Here we show that for continuously differentiable paths, the signature kernel solves a hyperbolic PDE and recognize the connection with a well known class of differential equations known in the literature as Goursat problems. This Goursat PDE only depends on the increments of the input sequences, does not require the explicit computation of signatures and can be solved efficiently using state-of-the-art hyperbolic PDE numerical solvers, giving a kernel trick for the untruncated signature kernel, with the same raw complexity as the method from \cite{kiraly2019kernels}, but with the advantage that the PDE numerical scheme is well suited for \texttt{GPU} parallelization, which effectively reduces the complexity by a full order of magnitude in the length of the input sequences. In addition, we extend the previous analysis to the space of geometric rough paths and establish, using classical results from rough path theory, that the rough version of the signature kernel solves a rough integral equation analogous to the aforementioned Goursat problem. Finally, we empirically demonstrate the effectiveness of this PDE kernel as a machine learning tool in various data science applications dealing with sequential data. We release the library \texttt{sigkernel} publicly available at \url{https://github.com/crispitagorico/sigkernel}.
\end{abstract}



\section{Introduction}\label{sec:intro}

Nowadays, sequential data is being produced and stored at an unprecedented rate. Examples include daily fluctuations of asset prices in the stock market, medical and biological records, readings from mobile apps, weather measurements etc. An efficient learning algorithm must be able to handle data streams that are often irregularly sampled and/or partially observed and at the same time scale well with a high number of channels. \smallbreak

An important obstacle that most machine learning models have to face is the potential symmetry present in the data. In computer vision for example, a good model should be able to recognize an image even if the latter is rotated by a certain angle. The $3$D rotation group, often denoted $SO(3)$, is low dimensional ($3$), therefore it is relatively easy to add components to a model that build a rotation invariance. However, when dealing with sequential data one is confronted with a much bigger (infinite dimensional) group of symmetries given by all \emph{reparametrizations} of a path\footnote{Or its time-augmented version.} (i.e. continuous and increasing surjections from the time domain of the path to itself). For example, consider the reparametrization $\phi: [0,1] \to [0,1]$ given by $\phi(t) = t^2$ and the path $\gamma : [0,1] \to \mathbb{R}^2$ defined by $\gamma_t = (\gamma^x_t, \gamma^y_t)$ where $\gamma^x_t=\cos(10t)$ and $\gamma^y_t=\sin(3t)$. As it is clearly depicted in \cref{fig:reparam}, both channels ($\gamma^x,\gamma^y$) of $\gamma$ are individually affected by the reparametrization $\phi$, but the shape of the curve $\gamma$ is left unchanged.
\begin{figure}[h]
    \centering
    \makebox[\textwidth]{\includegraphics[scale=0.45]{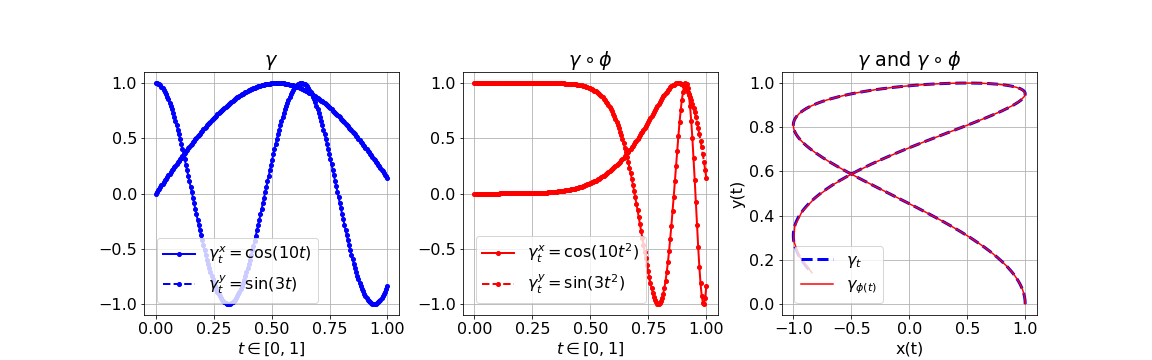}}
    \caption{{\small On the \textbf{left} are the individual channels $(\gamma^x, \gamma^y)$ of a $2$D paths $\gamma$. In the \textbf{middle} are the channels reparametrized under $\phi: t \mapsto t^2$. On the \textbf{right} are the path $\gamma$ and its reparametrized version $\gamma \circ \phi$. The two curves overlap, meaning that the reparametrization $\phi$ represents irrelevant information if one is interested in understanding the shape of $\gamma$.}}
    \label{fig:reparam}
\end{figure}

\begin{definition}\label{def:tensoralgebras}
Let $V$ be a Banach space. The spaces of formal polynomials and formal power series over $V$ are defined respectively as
\begin{equation}
    T(V) = \bigoplus_{k=0}^\infty V^{\otimes k} \quad \text{ and } \quad T((V)) = \prod_{k=0}^\infty V^{\otimes k}
\end{equation}
where $\otimes$ denotes the (classical) tensor product of vector spaces. Both $T(V)$ and $T((V))$ can be endowed with the operations of addition $+$ and multiplication $\otimes$ defined for any two elements $A = (a_{0}, a_{1}, \cdots)$ and $B = (b_{0}, b_{1}, \cdots)$ respectively as
\begin{align}
A + B &= \left(a_{0} + b_{0}, a_{1} + b_{1}, ... \right) \\ 
A \otimes B &= \left(c_{0}, c_{1}, c_{2}, ... \right), \ \ \text{ where } \ \ V^{\otimes k} \ni c_k = \sum_{i=0}^ka_i\otimes b_{k-i}, \ \ \forall k\geq 0 \label{eqn:tensor_product}
\end{align}
When endowed with these two operations and the natural action of $\mathbb{R}$ by $\lambda A = (\lambda a_0, \lambda a_1, ...)$, $T((V))$ becomes a real, non-commutative unital algebra with unit $\mathbf{1}=(1,0,0,...)$ called the tensor algebra. The truncated tensor algebra over $V$ of order $N \in \mathbb{N}$ is defined as the quotient $T^N(V) = T((V))/T_N$ by the ideal
\begin{equation}
    T_N=\{A=(a_0,a_1,...) \in T((V)):a_0=...=a_N=0\}
\end{equation} \medbreak
\end{definition}

\begin{definition}\label{def:signature}
Let $I \subset \mathbb{R}$ be a compact interval, $V$ a Banach space and let $x: I \to V$ be a continuous path of finite $p$-variation (\cref{def:p-variation}) with $p<2$. For any $s,t \in I$ such that $s \leq t$, the signature $S(x)_{[s,t]} \in T((V))$ of the path $x$ over the sub-interval $[s,t]$ is defined as the following infinite collection of iterated integrals
\begin{equation}\label{eqn:signature}
S(x)_{[s,t]} = \left(1, \underset{s<u_1<t}{\int} dx_{u_1}, ..., \underset{s<u_1<...<u_k<t}{\int ... \int}dx_{u_1}\otimes...\otimes dx_{u_k}, ... \right)
\end{equation}
\end{definition}

The \emph{signature} (\cref{def:signature}) of a path $x$ is invariant under reparametrization (i.e. $S(x) = S(x \circ \phi)$), therefore it acts on $x$ as a filter that systematically removes this troublesome, infinite dimensional group of symmetries.  Furthermore, it turns out that linear functionals acting on the range of the signature form an algebra (with pointwise multiplication) and separate points \cite[chapter 2]{lyons2004differential}. Hence, by the \emph{Stone-Weierstrass theorem}, for any compact set $C$ of continuous paths of bounded variation, the set of linear functionals on signatures of paths from $C$ is dense in the set of continuous real-valued functions on $C$. These two properties make the signature an ideal \emph{feature map} for data streams \cite{lyons2014roughICM}. \smallbreak

For any path $x$ of finite $p$-variation ($p>1$) the terms in the signature  \emph{decay factorially} according to the following uniform estimate \cite[Lemma 5.1]{lyons2014roughICM}
\begin{equation}\label{eqn:factorial_decay}
    \norm{\underset{s<u_1<...<u_k<t}{\int ... \int}dx_{u_1}\otimes...\otimes dx_{u_k}}_{V^{\otimes k}} \leq \frac{\norm{x}^k_{p,[s,t]}}{k!}
\end{equation}
where $\norm{\cdot}_{V^{\otimes k}}$ denotes any norm on $V^{\otimes k}$ and $\norm{x}_{p,[s,t]}$ denotes the $p$-variation of the path $x$ restricted to the interval $[s,t]$. Therefore, the collection of iterated integrals in the signature is \emph{graded}. This grading allows one to \emph{truncate} the signature at a finite level $N \in \mathbb{N}$ and consider only a finite collection of integrals as features extracted from $x$
\begin{equation}\label{eqn:truncated_signature}
    S^N(x)_{[s,t]} = \left(1, \underset{s<u_1<t}{\int} dx_{u_1}, ..., \underset{s<u_1<...<u_N<t}{\int ... \int}dx_{u_1}\otimes...\otimes dx_{u_N}\right) \in T^N(V)
\end{equation}
However, it is clear that the (truncated) signature has an exponential growth in the number of features, limiting its successful direct usage to machine learning applications where the ambient space of the data streams is relatively low dimensional \cite{arribas2018signature, morrill2020utilization, bonnier2019deep, yang2017developing, moore2019using, cochrane2021sk, lyons2019numerical}. \smallbreak

\emph{Kernel methods} \cite{hofmann2008kernel} have shown to be efficient learning techniques in situations where inputs are non-euclidean, high-dimensional (not necessarily sequential) and the number of training instances is limited \cite{shawe2004kernel} so that deep learning methods cannot be easily deployed. Most kernels that are used in practice can be computed efficiently without referring back to the corresponding feature map, a mechanism known as \emph{kernel trick}. When the data is sequential, the design of appropriate kernel functions is a notably challenging task \cite{cuturi2011fast}. In \cite{kiraly2019kernels} the authors introduce the \emph{truncated signature kernel} as the inner product of two truncated signatures and propose an efficient algorithm to compute this kernel starting from any ``static'' kernel on the ambient space of the input paths. \smallbreak

One of our goals will be to extend the results in \cite{kiraly2019kernels} and consider the \emph{ (untruncated) signature kernel} as an inner product of two (untruncated) signatures. For this, leveraging a fundamental property of the signature (\cref{thm:universal}), we prove in \cref{sec:BV} that if the two input paths are continuously differentiable then the signature kernel is the solution of a linear, second order, hyperbolic \emph{partial differential equation} (PDE). In \cref{sec:Goursat} we recognize the connection between the signature kernel PDE and a class of differential equations known in the literature as \emph{Goursat problems} \cite{goursat1916course}. This PDE represents effectively a kernel trick for the signature kernel and can be efficiently solved numerically leveraging any state-of-the-art hyperbolic PDE solvers; we provide ourselves a competitive finite difference explicit scheme and demonstrate the improvement in computational performance over existing approximation methods. In \cref{sec:geom} we extend the previous analysis to the much broader class of \emph{geometric rough paths}, and show that in this case the signature kernel satisfies an integral equation analogous to the aforementioned Goursat PDE. 
Finally in \cref{sec:applications}, we empirically demonstrate the effectiveness of the signature kernel on various data science applications dealing with sequential data. \smallbreak

We release the python library \texttt{sigkernel} implementing our signature PDE kernel and various other functionalities deriving from it. All the experiments presented in this paper are reproducible following the instructions in \url{https://github.com/crispitagorico/sigkernel}.

\begin{remark}
    A concise summary of rough path theory, covering the material necessary to follow the proofs in \cref{sec:geom}, is presented in \cref{appendix:RPT}. We note that an efficient algorithm for computing the truncated signature kernel was derived in \cite{kiraly2019kernels} and then used in \cite{Tth2019VariationalGP} in the context of Gaussian processes indexed on time series. Finally, we note that the article \cite{chevyrev2018signature} first treated the truncated signature kernel in the case of branched rough paths. Integration of two parameters rough integrals is also discussed in \cite{chouk2014rough}. \medbreak
\end{remark}

\section{The Signature Kernel is the solution of a hyperbolic PDE}\label{sec:BV}

In this section we present our main result, notably that the signature kernel evaluated at two continuously differentiable paths is the solution of a hyperbolic PDE. Throughout this section, we will denote by $C^1(I,V)$ the space of continuously differentiable paths defined over an interval $I=[u,u']$ and with values on a Banach space $V$. We will also use the lighter notation $S(x)_t$ to denote the signature of a path $x$ over the interval $[u,t]$, for any $t \in I$.  \medbreak

\begin{definition}
Let $V$ be a $d$-dimensional Banach space with canonical basis $\{e_1,...,e_d\}$. It is easy to verify that for any $k \geq 1$ the elements
\begin{equation}
    \{e_{i_1} \otimes ... \otimes e_{i_k} : (i_1, ... , i_k) \in \{1 , ..., d\}^k\}
\end{equation}
form a basis of $V^{\otimes k}$. Consider the inner product on $V^{\otimes k}$ defined on basis elements as
\begin{equation}
    \langle e_{i_1} \otimes ... \otimes e_{i_k}, e_{j_1} \otimes ... \otimes e_{j_k} \rangle_{V^{\otimes k}} = \langle e_{i_1}, e_{j_1} \rangle_V ... \langle e_{i_k}, e_{j_k} \rangle_V
\end{equation}

The inner product $\langle \cdot, \cdot \rangle_{V^{\otimes k}}$ can be extended by linearity to an inner product on $T(V)$ defined for any $A = (a_0, a_1, ...), B=(b_0,b_1,...)$ in $T(V)$ as
\begin{equation}
    \langle A, B \rangle = \sum_{k=0}^\infty \langle a_k, b_k \rangle_{V^{\otimes k}}
\end{equation}
\end{definition}

\begin{remark}
    It is easy to verify that the space $T((V))$ has the following algebraic property that we will refer to as \emph{coproduct property}. Let $m,n \in \mathbb{N}$ be two positive integers and consider any two elements $A,B \in T((V))$ and any two basis elements $e_{i_1}, e_{i_2} \in V$ seen as elements of $T((V))$, i.e. as $(0,e_{i_1}, 0,...)$ and $(0,e_{i_2}, 0,...)$ respectively. Then the following identity holds
    \begin{equation}\label{eqn:property}
        \langle A \otimes e_{i_1}, B \otimes e_{i_2} \rangle = \langle A, B \rangle \langle e_{i_1}, e_{i_2} \rangle_V 
    \end{equation}
\end{remark}

As anticipated in the introduction, the signature has a fundamental characterisation in terms of \emph{controlled differential equations} (CDEs) (see \cref{appendix:CDE} for a brief account on CDEs). In effect, the signature solves the universal differential equation stated in the next theorem, and therefore it can be equivalently defined as the \emph{non-commutative exponential}.

\begin{theorem}\cite[Lemma 2.10]{lyons2004differential}\label{thm:universal}
Let $x: I \to V$ be a continous path of finite $p$-variation for $p<2$ and $A = (a_0, a_1, ...) \in T(V)$. Consider the vector field $f:T(V) \to L(V,T(V))$\footnote{$L(V,T(V))$ denotes the space of bounded linear maps from $V$ to $T(V)$.}
\begin{equation}
    f(A)(v) = A \otimes v = (0,a_0\otimes v,a_1 \otimes v, ...) 
\end{equation}
Then, the unique solution to the following controlled differential equation 
\begin{equation}\label{eqn:univ_}
    dS_t = f(S_t)dx_t, \quad S_0=(1,0,0,...)
\end{equation}
is the signature $S(x)_t$ of the path $x$. \cref{eqn:univ_} can be formally rewritten as
    \begin{equation}\label{eqn:univ}
        dS(x)_t = S(x)_t \otimes dx_t, \hspace{0.5cm} S(x)_0=(1,0,0,...)
    \end{equation}
    which explains why the signature of $x$ can be described as its non-commutative exponential. \smallbreak
\end{theorem}

\subsection{The Signature Kernel PDE}

\begin{definition}\label{def:sig_kernel}
Let $I=[u, u']$ and $J=[v, v']$ be two compact intervals and let $x \in C^1(I,V)$ and $y \in C^1(J,V)$. The signature kernel $k_{x,y}:  I \times J \to \mathbb{R}$ is defined as
\begin{equation}
    k_{x,y}(s,t) = \langle S(x)_s, S(y)_t \rangle
\end{equation}
\end{definition}

The following is the main result of this section; it unveils a simple relation between the signature kernel and a class of hyperbolic PDEs.\smallbreak

\begin{theorem}\label{thm:PDE}
    Let $I=[u, u']$ and $J=[v, v']$ be two compact intervals and let $x \in C^1(I,V)$ and $y \in C^1(J,V)$. The signature kernel $k_{x,y}$ is a solution of the following linear, second order, hyperbolic PDE
    \begin{equation}\label{eqn:PDE}
       \frac{\partial^2 k_{x,y}}{\partial s \partial t} = \langle \dot x_s, \dot y_t \rangle_V k_{x,y}, \quad k_{x, y}(u,\cdot) = k_{x,y}(\cdot,v) = 1
    \end{equation}
    where $\dot x_s = \frac{dx_p}{dp}\big|_{p=s}$, $\dot y_t = \frac{dx_q}{dq}\big|_{q=t}$ are the derivatives of $x$ and $y$ at time $s$ and $t$ respectively. \medbreak
\end{theorem}

\begin{proof}
Clearly, for any $t \in J$ one has
\begin{align*}
    k_{x,y}(u,t) & = \left\langle S(x)_{[u,u]},  S(y)_{[v,t]} \right\rangle \\
    &= \left\langle (1,0,...), S(y)_{[v,t]} \right\rangle \\
    & = 1
\end{align*}
and similarly $k_{x,y}(s,v)=1$ for any $s \in I$. Recall that the signature of a path $x:I \to V$ satisfies \cref{eqn:univ}, which is equivalent to the following integral equation
\begin{equation*}
    S(x)_s = \mathbf{1} + \int_{p=u}^sS(x)_p \otimes dx_p
\end{equation*}
where $\mathbf{1} = (1,0,0,...)$. Similarly for $S(y)_t$. Hence, we can compute
\begin{align}
    k_{x,y}(s,t) &= \langle S(x)_s,  S(y)_t \rangle \nonumber \\
    &=\Big\langle \mathbf{1} + \int_{p=u}^{s} S(x)_p \otimes dx_p, \ \mathbf{1} + \int_{q=v}^{t} S(y)_q \otimes dy_q \Big\rangle && \quad \text{(\cref{thm:universal})} \nonumber\\
    &= 1 + \Big\langle \int_{p=u}^{s} S(x)_p \otimes \dot x_p dp, \ \int_{q=v}^{t} S(y)_q \otimes \dot y_q dq \Big\rangle  &&  \quad \text{(differentiability)} \nonumber \\
    & = 1 + \int_{p=u}^{s} \int_{q=v}^{t} \langle S(x)_p \otimes \dot x_p, \ S(y)_q \otimes \dot y_q \rangle dpdq && \quad \text{(linearity)} \label{eqn:linearity}\\
    &= 1 + \int_{p=u}^{s} \int_{q=v}^{t} \langle S(x)_p, S(y)_q \rangle \langle \dot x_p, \dot y_q \rangle_V dpdq && \quad \text{(coproduct property \cref{eqn:property})} \nonumber \\
    & = 1 + \int_{p=u}^{s} \int_{q=v}^{t} k_{x,y}(p,q) \langle \dot x_p, \dot y_q \rangle_V dpdq && \quad \text{(definition of $k_{x,y}$)} \nonumber
\end{align}
Note that the inner product and the double integral can be interchanged in \cref{eqn:linearity} because of the factorial decay (\cref{eqn:factorial_decay}) of the terms in the signature. By the \emph{fundamental theorem of calculus} we can differentiate firstly with respect to $s$
\begin{equation*}
\frac{\partial k_{x,y}(s,t)}{\partial s} = \int_{q=v}^{t} k_{x,y}(s,q) \langle \dot x_s, \dot y_q \rangle_V dq
\end{equation*}
and then with respect to $t$ to obtain the PDE \cref{eqn:PDE}
\begin{equation*}
\frac{\partial^2 k_{x,y}(s,t)}{\partial s \partial t} = \langle \dot x_s, \dot y_t \rangle_V k_{x,y}(s,t) 
\end{equation*}
\end{proof}

\begin{remark}
In \cref{thm:PDE} we have assumed the two input paths $x,y$ to be of class $C^1$. However, one can lower this regularity assumption and consider two continuous paths $x,y$ of bounded variation and obtain the following integral equation
\begin{equation}
    k_{x,y}(s,t) = 1 + \int_{p=u}^s \int_{q=v}^t \langle S(x)_p, S(y)_q \rangle \langle dx_p, dx_q \rangle_V
\end{equation}
where $\langle dx_p, dx_q \rangle_V$ is a quantity we are going to give meaning in \cref{sec:geom}, when we will consider the broader class of \emph{geometric rough paths}. We note that one can make sense of the PDE \cref{eqn:PDE} in \cref{thm:PDE} for piecewise $C^1$ paths.
\end{remark}


Sequential information often arrives in the form of complex data streams taking their values in non-trivial \emph{ambient spaces}. A good learning strategy would be to first \emph{lift} the underlying ambient space to a (possibly infinite dimensional) \emph{feature space} by means of a \emph{feature map} on static data (RBF, Matern etc.), and then consider the signature kernel of the lifted paths as the final learning tool \cite{kiraly2019kernels}. A question that naturally arises is whether one can compute the signature PDE kernel of the lifted paths from the static kernel associated to this feature map. \cite{kiraly2019kernels} propose an algorithm to perform this procedure. Next, we provide an explanation of this procedure in the language of Banach spaces and PDEs.       

\subsection{The signature PDE kernel from static kernels on the ambient space}
A kernel can be identified with a pair of embeddings of a set $\mathcal{X}$ into a Banach space
$E$ and its topological dual $E^{\ast};$ we denote this pair of maps by
$\phi:\mathcal{X}\rightarrow E$ and $\psi:\mathcal{X}\rightarrow E^{\ast}.$ A kernel induces a
function $\kappa:\mathcal{X}\times \mathcal{X}\rightarrow\mathbb{R}$ through the natural pairing
between a Banach space and its dual
\begin{equation}
    \kappa(a,b) := (\phi(a), \psi(b))_E, \quad \text{for all } a,b \in \mathcal{X}
\end{equation}
Commonly, $E$ is assumed to be a Hilbert space, in which case $\psi$ can be taken to be the composition
$e\circ\phi$ where $e:E\rightarrow E^{\ast}$ is the canonical isomorphism coming from the \emph{Riesz representation theorem}, yielding $\kappa (a,b) = \langle \phi(a),\phi(b) \rangle_{E}$. It is unnecessary however for the general picture for $E$ to be a Hilbert space. In the general framework, a given pair of paths $x:I  \to \mathcal{X}$ and $y : J  \to \mathcal{X}$, with $I=[u,u'], J=[v,v']$, can be lifted to paths on $E$ and $E^{\ast}$ respectively as follows
\begin{equation}\label{eqn:seq}
    X_s = \phi(x_s), \quad Y_t=\psi(y_t),  \quad \text{for all } s\in I, t \in J
\end{equation}
If we assume that $X$ and $Y$ are continuous and have bounded variation, then their signatures are well defined and belong to $T(E)$, which is again a Banach space with $T^N(E)^{\ast}\cong T^N(E^{\ast})$ for any $N \geq 1$ \cite{lyons2004differential}. Hence, starting with a kernel $\kappa$ on $\mathcal{X}$, the signature kernel is well-defined
\begin{equation}
    k_{X,Y}(s,t) = (S(\phi \circ x)_s, S(\psi \circ y)_t)_{T(E)} = (S(X)_s, S(Y)_t)_{T(E)}
\end{equation}
If furthermore we assume that the lifted paths $X,Y$ are of class $C^1$ \cref{thm:PDE} applies, yielding the following PDE
\begin{equation}\label{eqn:PDE_lifted}
   \frac{\partial^2 k_{X,Y}}{\partial s \partial t} = ( \dot X_s, \dot Y_t)_E \ k_{X,Y}, \quad k_{X,Y}(u,\cdot) = k_{X,Y}(\cdot,v) = 1
\end{equation}
With a first order finite difference approximations for the derivatives $\dot X_s, \dot Y_t$, the PDE \cref{eqn:PDE_lifted} can be entirely expressed in terms of the static kernel $\kappa$ and underlying paths $x,y$ as follows
\begin{align}
   \frac{\partial^2 k_{X,Y}}{\partial s \partial t} & = ((X_s,Y_t)_E  - (X_{s-1}, Y_t)_E - (X_s,Y_{t-1})_E + (X_{s-1}, Y_{t-1})_E) k_{X,Y} \nonumber \\
   & = (\kappa(x_s,y_t)  - \kappa(x_{s-1}, x_t) - \kappa(x_s,y_{t-1}) + \kappa(x_{s-1}, y_{t-1})) k_{X,Y}
\end{align}\medbreak

\begin{remark}
We note that it is possible to establish our Goursat PDE \cref{eqn:PDE} from the results \cite[Proposition 4.7 p. 16]{kiraly2016kernels}, \cite[Theorem 4., Appendix A]{kiraly2016kernels}, but this requires additional arguments as we shall explain next. Using their notation, given two differentiable paths $\sigma, \tau$, the truncated (at level $M$) signature kernel $k^{\oplus}_{\leq M}(\sigma, \tau)$ satisfies the following equation
\begin{align*}\small
    k^{\oplus}_{\leq M}(\sigma, \tau)_{u,v} = 1 + &\underset{(s_1, t_1) \in (0,u) \times (0,v)}{\int \int}\Big(1 + ... \underset{(s_M,t_M) \in (0,s_{M-1}) \times (0,t_{M-1})}{\int \int} d\kappa_{\sigma,\tau}(s_M,t_M) \Big)d\kappa_{\sigma,\tau}(s_1,t_1)
\end{align*}
where $d\kappa_{\sigma,\tau}(s,t) = k(\dot x_s, \dot y_t) dsdt$ and where $k$ is a kernel on the ambient space of the paths. The first step towards a PDE is to realise that the first integrand in the last equation is itself the truncated signature kernel, truncated at level $M-1$, yielding the expression
\begin{align}
    k^{\oplus}_{\leq M}(\sigma, \tau)_{u,v} = 1 + &\underset{(s_1, t_1) \in (0,u) \times (0,v)}{\int \int} k^{\oplus}_{\leq M-1}(\sigma, \tau)_{s_1,t_1} d\kappa_{\sigma,\tau}(s_1,t_1) \label{eqn:Franz2}
\end{align}
The untruncated signature kernel is obtained by taking the limit in \cref{eqn:Franz2} when $M \to \infty$. The factorial decay in the terms of the signature yields uniform convergence of this limiting process. Two uniformly convergent sequences of functions that are equal for all finite levels $M$ they are also equal in the limit, which implies 
\begin{align}
    k^{\oplus}_{\leq \infty}(\sigma, \tau)_{u,v} = 1 + &\underset{(s_1, t_1) \in (0,u) \times (0,v)}{\int \int} k^{\oplus}_{\leq \infty}(\sigma, \tau)_{s_1,t_1} d\kappa_{\sigma,\tau}(s_1,t_1) \label{eqn:Franz3}
\end{align}
Finally one substitutes $d\kappa_{\sigma,\tau}(s,t) = k(\dot x_s, \dot y_t) dsdt$. At that point (and similarly to our argument in Theorem 2.5) one can differentiate both sides of \cref{eqn:Franz3} to get the Goursat PDE.    
\end{remark}

In the next section we recognize the link between \cref{eqn:PDE} and a class of differential equation known in the literature as Goursat problems and propose a competitve numerical solver for our specific PDE.

\section{A Goursat problem}\label{sec:Goursat}

\cref{eqn:PDE} is an instance of a Goursat problem, which is a class of hyperbolic PDEs introduced in \cite{goursat1916course}. The PDE (\ref{eqn:PDE}) is defined on the bounded domain 
\begin{equation}
    \mathcal{D} := \{(s,t) \mid u\leq s \leq u', v \leq t \leq v'\} \subset I \times J\\
\end{equation}
and its existence and uniqueness (for paths of class $C^1$) are guaranteed by setting the functions $C_1=C_2=C_4=0$ and $C_3(s,t)=\langle \dot x_s, \dot y_t \rangle_V$ in the following result.

\begin{theorem}\cite[Theorems 2 \& 4]{lees1960goursat}\label{thm:goursat}
    Let $\sigma : I \to \mathbb{R}$ and $\tau: J \to \mathbb{R}$ be two absolutely continuous functions whose first derivatives are square integrable and such that $\sigma(u) = \tau(v)$. Let $C_1,C_2,C_3:\mathcal{D} \to \mathbb{R}$ be a bounded and measurable over $\mathcal{D}$ and $C_4:\mathcal{D} \to \mathbb{R}$ be square integrable. Then there exists a unique function $z : \mathcal{D} \to \mathbb{R}$ such that $z(s,v)=\sigma(s), z(u,t)=\tau(t)$ and (almost everywhere on $\mathcal{D}$)
    \begin{equation}\label{eqn:Goursat}
        \frac{\partial^2 z}{\partial s \partial t} = C_1(s,t)\frac{\partial z}{\partial s} + C_2(s,t)\frac{\partial z}{\partial t} + C_3(s,t)z + C_4(s,t)
    \end{equation}
If in addition $C_i \in C^{p-1}(\mathcal{D})$ ($i=1,2,3,4$) and $\sigma$ and $\tau$ are $C^p$, then the unique solution $z:\mathcal{D} \to \mathbb{R}$ of the Goursat problem is of class $C^p$.
\end{theorem}

In the case of the signature PDE kernel, if the two input paths $x,y$ are of class $C^p$, then their derivatives will be of class $C^{p-1}$, and therefore \cref{thm:goursat} implies that the solution $k_{x,y}$ of the PDE \cref{eqn:PDE} will be of class $C^p$.

\subsection{Finite difference approximation}

In this section, we propose a numerical method based on an explicit finite difference scheme to approximate the solution of the Goursat PDE \cref{eqn:PDE}. To simplify the notation, we consider the case where $V=\mathbb{R}^d$.  If $x$ and $y$ are piecewise linear, then the PDE \cref{eqn:PDE} becomes
\begin{equation}\label{eqn:constant_Goursat}
\frac{\partial^2 k_{x,y}}{\partial s \partial t} = C_3 k_{x,y\,},
\end{equation}
on each domain $\mathcal{D}_{ij} = \{(s,t) \mid u_i\leq s \leq u_{i+1}, v_j \leq t \leq v_{j+1}\}$ where $C_3 =\langle \dot x_s, \dot y_t \rangle_V$ is constant. In integral form, the PDE \cref{eqn:constant_Goursat} can be written as
\begin{equation}\label{eqn:integral_Goursat}
k_{x,y}(s,t) = k_{x,y}(s,v) + k_{x,y}(u,t) - k_{x,y}(u,v) + C_3\int_u^s\int_v^t k_{x,y}(r,w)\,dr\,dw,
\end{equation}
for $(s,t),(u,v)\in\mathcal{D}_{ij}$ with $u \leq s$ and $v \leq t$. By approximating the double integral in \cref{eqn:integral_Goursat}, we can derive the following numerical explicit scheme
\begin{align}\label{eqn:explicit_Goursat}
k_{x,y}(s,t) & \approx k_{x,y}(s,v) + k_{x,y}(u,t) - k_{x,y}(u,v)\\[3pt]
    &\hspace{11mm} + \frac{1}{2}C_3\big(k_{x,y}(s,v) + k_{x,y}(u,t)\big)(u-s)(t-v). \nonumber
\end{align}

\begin{remark}
    An implicit scheme can be obtained by estimating \cref{eqn:integral_Goursat} with all four values of $k_{x,y}$ as follows
    \begin{align}\label{eqn:implicit_scheme}
    k_{x,y}(s,t) &\approx k_{x,y}(s,v) + k_{x,y}(u,t) - k_{x,y}(u,v)\\[3pt]
    &\hspace{11mm} + \frac{1}{4}C_3\big(k_{x,y}(u,v) + k_{x,y}(s,v) + k_{x,y}(u,t) + k_{x,y}(s,t)\big)(u-s)(t-v). \nonumber
    \end{align}
\end{remark}\medbreak

As one might expect, more sophisticated approximations can be derived by applying higher order quadrature methods to the double integral in (\ref{eqn:integral_Goursat}) (see \cite{day1966finitediff, wazwaz1993finitediff} for specific examples).\medbreak

Let $\mathcal{D}_I = \{u=u_0<u_1<...<u_{m-1}<u_m=u'\}$ be a partition of the interval $I$ and $\mathcal{D}_J = \{v=v_0<v_1<...<v_{n-1}<v_n=v'\}$ be a partition of the interval $J$. Using the above, we can define \textit{finite difference schemes} on the grid $P_0:=\mathcal{D}_I \times \mathcal{D}_J$ (and its dyadic refinements). \medbreak

\begin{figure}[h]
\centering
\centerline{\includegraphics[scale=0.45]{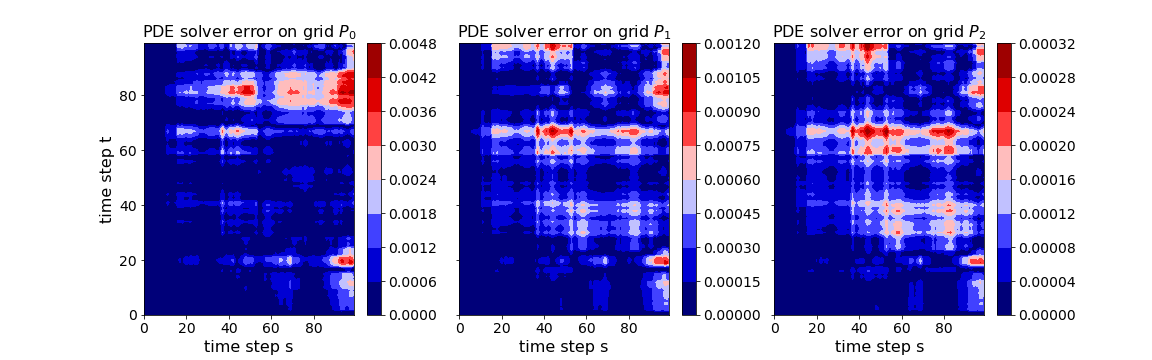}}
\caption{{\small Example of error distribution of $k_{x,y}(s,t)$ on the grids $P_0, P_1, P_2$. The discretization is roughly four times more accurate on $P_1$ than on $P_0$ (as expected by \cref{thm:error}).}}
\label{fig:error}
\end{figure}

\begin{definition}\label{def:fintite_diff}
For $\lambda\in\{0,1,2,\cdots\}$, we define the grid $P_\lambda$ as the dyadic refinement of $P_0$
such that $P_\lambda \cap ([u_i, u_{i+1}]\times [v_i, v_{i+1}]) = \{u_i + k\,2^{-\lambda}(u_{i+1} - u_i), v_j + l\,2^{-\lambda}(v_{j+1} - v_j)\}_{0\,\leq\, k, l\, \leq\, 2^\lambda}$.\smallbreak\noindent
\end{definition} \medbreak

On the grid $P_\lambda = \{(s_i, t_j)\}_{0\,\leq\, i\, \leq\, 2^\lambda n,\, 0\,\leq\,j \, \leq\, 2^\lambda m}\,$, we define the following explicit finite difference scheme for the PDE \cref{eqn:PDE}
\begin{align}\label{eqn:explicit_finite_diff}
    \hat{k}(s_{i+1}, t_{j+1}) & = \hat{k}(s_{i+1}, t_j) + \hat{k}(s_i, t_{j+1}) - \hat{k}(s_i, t_j)\\
    &\hspace{10mm} + \frac{1}{2}\langle x_{s_{i+1}} - x_{s_i}, y_{t_{j+1}} - y_{t_j} \rangle\big(\hat{k}(s_{i+1}, t_j) + \hat{k}(s_i, t_{j+1})\big),\nonumber\\[3pt]
    \hat{k}(s_0, \cdot\,) & = \hat{k}(\,\cdot, t_0) = 1, \nonumber
\end{align} \medbreak

\begin{remark}
If $x$ and $y$ are piecewise linear paths with respect to the coarsest grid $P_{0\,}$
then $\langle x_{s_{i+1}} - x_{s_i}, y_{t_{j+1}} - y_{t_j} \rangle = \frac{1}{2^{2\lambda}}\langle x_{u_{p+1}} - x_{u_p}, y_{v_{q+1}} - y_{v_q}\rangle$ for some $0\leq p < n$ and $0\leq q < m$.
\end{remark}\medbreak


The explicit finite differences scheme \cref{eqn:explicit_finite_diff} has a time complexity of $O\big(d^2\,2^{2\lambda}\,mn\big)$ on the grid $P_\lambda$, where $d$ is the dimension of the input streams $x,y$ and $m,n$ denote their respective lengths. \cref{thm:error} (which is the proved in \cref{append:error}) ensures that by refining the discretization of the grid used to approximate the PDE, we get convergence to the true value. In practice we found that provided the input paths are rescaled so that their maximum value across all times and all dimensions is not too large ($\approx 1$), coarse partitioning choices such as $P_0$ or $P_1$ are sufficient to obtain a highly accurate approximation, as shown in \cref{fig:error}. \medbreak

\begin{theorem}[Global error estimate, \cref{append:error}]\label{thm:error}
Let $\widetilde{k}$ be a numerical solution obtained by applying one of the proposed finite difference schemes (\cref{eqn:explicit_finite_diff}) to the Goursat problem \cref{eqn:PDE} on $P_\lambda$ where $x$ and $y$ are piecewise linear with respect to the grids $\mathcal{D}_I$ and $\mathcal{D}_J$. In particular, we are assuming there exists a constant $M$, that is independent of $\lambda$, such that
\begin{equation}
\sup_{\mathcal{D}}|\langle \dot x_s, \dot y_t \rangle| < M.
\end{equation}
Then there exists a constant $K > 0$ depending on $M$ and $k_{x,y}$, but independent of $\lambda$, such that
\begin{equation}
\sup_{\mathcal{D}}\big|k_{x,y}(s, t) - \widetilde{k}(s, t)\big| \leq \frac{K}{2^{2\lambda}},\ \ \ \text{for all $\lambda \geq 0$}
\end{equation} 
\end{theorem}\medbreak

\subsubsection{GPU implementation of the Goursat PDE}

As mentioned earlier, the time complexity for one signature PDE kernel evaluation is $\mathcal{O}(d\ell^2)$ on $P_0$, where $d$ is the number of channels of the input time series and $\ell$ is their (maximum) length. Therefore, the complexity is quadratic in the length of the time series, which makes kernel evaluations computationally expensive for long time series. This also holds for the algorithm proposed in \cite{kiraly2019kernels}. However, it is possible to parallelize the PDE solver by observing that instead of solving the PDE in row or column order, we can update the antidiagonals of the solution grid: each cell on an antidiagonal can be updated in parallel as there is no data dependency between them. This breaks the quadratic complexity, that becomes linear in the length $\ell$, provided the number of threads in the \texttt{GPU} exceeds the size of the discretization. as shown in \cref{fig:comparison}. This parallelization is possible thanks to the ``PDE structure'' of the problem, representing a considerable computational gain of our algorithm compared to the one proposed by \cite{kiraly2019kernels}. We also note that the linear dependency on the number of channels $d$ of the input time series allows for the evaluation of the signature PDE kernel on time series with thousands of channels. Our library \texttt{sigkernel} offers the ability to evaluate kernels on a \texttt{CPU} using an optimized \texttt{cython} implementation as well as on \texttt{CUDA} if  \texttt{GPU}s are available to the user. \medbreak

\begin{figure}[ht]
\centering
\centerline{\includegraphics[scale=0.45]{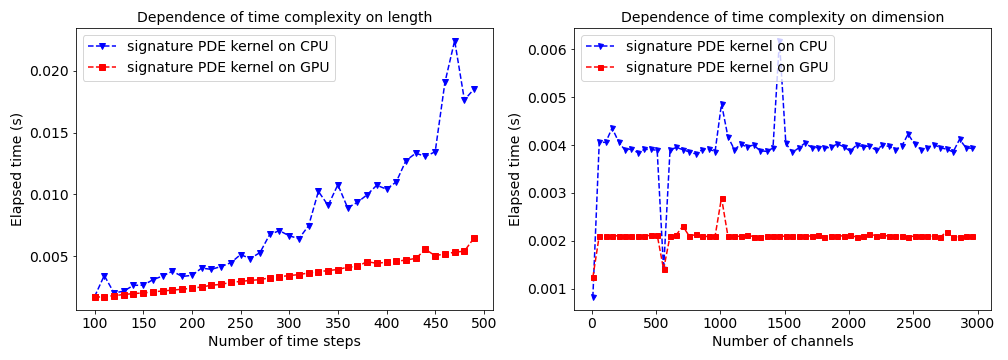}}
\caption{{\small Comparison of the elapsed time (s) to reach an accuracy of $10^{-3}$ from a target value obtained by solving the signature kernel PDE on a fine discretization grid ($P_5$). We simulate $N=5$ (piecewise linear interpolation of) Brownian paths at each run. On the \textbf{left} is the dependency on the length of two paths of dimension $d=2$. Note the complexity reduction from quadratic on \texttt{CPU} to linear on \texttt{GPU} (P100). On the \textbf{right} is the dependency on the dimension of two paths of length $\ell=10$.}}
\label{fig:comparison}
\end{figure}

In \cref{sec:applications} we will present various applications of the signature kernel to time series classification and regression problems. But first we continue our theoretical analysis and drop the smoothness assumption on the input paths $x,y$ and extending the definition of signature kernel to far less regular classes of paths, namely geometric rough paths. The need to investigate the rough version of the signature kernel can be motivated also from several practical viewpoints. For example, this kernel can be used to derive an (unbiased) estimator for the \emph{maximum mean discrepancy} (MMD) distance between distributions on path-space \cite[sections 7, 8]{chevyrev2018signature}. The MMD distance itself is useful to train models such as \emph{neural SDEs} \cite{tzen2019neural, kidger2021neural, li2020scalable, cuchiero2020generative}, that is  to fit neural SDEs to time series data. Since SDE solutions are geometric p-rough paths, \cref{thm:main} provides a candidate for the limiting kernel as the mesh size of the SDE discretization tends to zero. In particular, it guarantees that the signature kernel doesn't ``blow-up'' in the limit. Another area where the rough signature kernel could be relevant is quantitative finance, where \emph{rough volatility models} \cite{gatheral2018volatility, bayer2016pricing} try to calibrate differential equation driven by \emph{fractional Brownian motion}, which is a rough path if the \emph{Hurst exponent} $h \leq 2$.

\section{The signature kernel for geometric rough paths}\label{sec:geom}

Here we extend the notion of signature kernel developed in \cref{sec:BV} to the broader class of geometric rough paths. To follow the material presented in this section we assume that the reader has some level of familiarity with basic concepts of rough path theory. We provide a brief summary of this theory in \cref{appendix:RPT}. We begin by clarifying what we mean by signature of a geometric $p$-rough path.

\subsection{The signature of a geometric rough path}

\begin{definition}
The signature $S(X)$ of a geometric $p$-rough path $X \in G\Omega_p(V)$ (\cref{def:geom_rough_path}) controlled by a control (\cref{def:control}) $\omega$ is its unique extension to a multiplicative functional (\cref{def:multiplicative_functional}) on $T(V)$ as given by the Extension Theorem \ref{thm:extension_theorem}.

\end{definition}
From now on, we will denote by $G\Omega_p(V)$ the space of geometric $p$-rough paths over $V$. Because all the sums in $T(V)$ are finite, $(T(V), \langle \cdot, \cdot \rangle)$ is an inner product space. Hence, denoting by $\overline{T(V)}$ the completion of $T(V)$, $(\overline{T(V)}, \langle \cdot, \cdot \rangle)$ is a Hilbert space. Let $\norm{\cdot}$ be the norm on $\overline{T(V)}$ induced by the inner product $\langle \cdot, \cdot \rangle$, i.e. defined for any $A=(a_0,a_1,...) \in \overline{T(V)}$ as $\norm{A} = \sqrt{\sum_{k\geq 0}\norm{a_k}_{V^{\otimes k}}^2}$, where $\norm{\cdot}_{V^{\otimes k}}$ is the norm on $V^{\otimes k}$ induced by $\langle \cdot, \cdot \rangle_{V^{\otimes k}}$, for any $k \geq 0$. In summary, we have the following chain of inclusions
\begin{equation}
    T(V) \hookrightarrow \overline{T(V)} \hookrightarrow T((V))
\end{equation}
Note that $\overline{T(V)} = \{x \in T((V)) : || x || < \infty\}$. \medbreak

\begin{lemma}
Let $X$ be a geometric $p$-rough path $X \in G\Omega_p(V)$ defined over the simplex $\Delta_T$. Then, for any $(s,t) \in \Delta_T$ one has $S(X)_{s, t} \in \overline{T(V)}$. \medbreak
\end{lemma}

\begin{proof}
To prove the statement of the lemma it suffices to find a sequence of tensors $\{X^{(n)}_{s,t} \in T^{ k}(V)\}_{n \in \mathbb{N}}$ that convergences to $S(X_{s,t})$ in the $\| \cdot \|$-topology.
Setting $X^{(n)}_{s,t} = (1, X^1_{s,t}, ..., X^n_{s,t}, 0, ...)$, and using the bounds from the Extension Theorem \ref{thm:extension_theorem} we have
\begin{equation}
    \|S(X)_{s, t}\| = \sqrt{\sum_{k=0}^\infty \| X^{k}_{s, t}\|^2_{V^{\otimes k}}} \leq \sqrt{\sum_{k=0}^\infty \frac{\omega(s, t)^{2k/p}}{(\beta_p(k/p)!)^2}} \leq \sum_{k=0}^\infty \frac{\omega(s, t)^{k/p}}{\beta_p(k/p)!}
\end{equation}
which is clearly a convergent series because of the terms decay factorially, and $\forall (s, t) \in \Delta_I$
\begin{equation}
    \|X^{(n)}_{s, t} - S(X)_{s, t}\| = \sqrt{\sum_{k\geq n+1}^\infty \|X^k_{s,t}\|^2_{V^{\otimes k}}} \longrightarrow 0 \text{ as } n \to \infty
\end{equation}
\end{proof}

Next we present our second main result, that is we extend the definition of signature kernel to the space of geometric $p$-rough paths (\cref{def:geom_rough_path}) and show that this rough version of the signature kernel solves an iterated double integral equation of two one-forms analogous to the Goursat PDE \cref{eqn:PDE} presented in \cref{sec:BV}. 

\subsection{The signature kernel for geometric rough paths}

In what follows $\Delta_I, \Delta_J$ will denote the following two simplices
\begin{align}
    \Delta_I &= \{(s,t) \in [i_-,i_+]^2 : i_-\leq s \leq t \leq i_+\}\\
    \Delta_J &= \{(s,t) \in [j_-,j_+]^2 : j_- \leq s \leq t \leq j_+\}
\end{align}
where $i_-, i_+, j_-, j_+ \geq 0$ are positive scalars such that $i_-<i_+$ and $j_-<j_+$.

\begin{definition}\label{def:rough_kernel}
Let $p,q \geq 1$ be two scalars. Let $X \in G\Omega_p(V)$ and $Y \in G\Omega_q(V)$ be two geometric $p$- and $q$-rough paths respectively and controlled by two controls $\omega_X$ and $\omega_Y$ respectively. The (rough) signature kernel $K_{(s_1,s_2), (t_1,t_2)} : G\Omega_p(V) \times G\Omega_q(V) \to \mathbb{R}$ is defined for any $(s_1,s_2) \in \Delta_I$ and $(t_1,t_2) \in \Delta_J$ as follows
    \begin{equation}
      K_{(s_1,s_2), (t_1,t_2)}(X, Y) =  \big\langle S(X)_{s_1,s_2}, S(Y)_{t_1,t_2} \big\rangle
    \end{equation}
where the inner product is taken in $\overline{T(V)}$.
\end{definition}

\begin{remark}
    On the one hand, the signature kernel of \cref{def:sig_kernel} is configured to act on two time indices and is indexed on two paths. This choice was made in order to differentiate with respect to these and obtain the PDE \cref{eqn:PDE}. On the other hand, the rough signature kernel of \cref{def:rough_kernel} acts on two (rough) paths and is indexed on time indices. When dealing with highly oscillatory objects like rough paths studied in this section, one can't expect to obtain a PDE, as these paths are far from being differentiable (even locally). However, we will nonetheless be able to use a density argument to prove our main result (\cref{thm:main}).
\end{remark}

Next we show that the rough signature kernel is bounded and continuous.
	
\begin{lemma}\label{lemma:kernel}
    For any $(X, Y) \in G\Omega_p(V) \times G\Omega_q(V)$ and any $(s_1, s_2) \in \Delta_I, (t_1, t_2) \in \Delta_J$
    \begin{equation}
        \big\langle S(X)_{s_1, s_2}, S(Y)_{t_1, t_2} \big\rangle < +\infty    
    \end{equation}
    Furthermore, the rough signature kernel $K_{(s_1,s_2), (t_1,t_2)}$ is continuous with respect to the the product $p,q$-variation topology.
\end{lemma}
   
\begin{proof}
For any $(s_1,s_2) \in \Delta_I, (t_1,t_2) \in \Delta_J$ and by definition of the inner product $\langle \cdot, \cdot \rangle$ on $\overline{T(V)}$ we immediately have 
\begin{align*}
    \langle S(X_{s_1,s_2}), S(Y_{t_1, t_2}) \rangle &= \sum_{k=0}^\infty \langle X_{s_1, s_2}^k, Y_{t_1, t_2}^k \rangle_{V^{\otimes k}}\\
    &\leq \sum_{k=0}^\infty \|X_{s_1, s_2}^k\|_{V^{\otimes k}} \|Y_{t_1, t_2}^k\|_{V^{\otimes k}} \hspace{2cm} \text{ (Cauchy-Schwarz)} \\
    & \leq \sum_{k=0}^\infty \frac{\omega_X(s_1,s_2)^{k/p} \cdot \omega_Y(t_1,t_2)^{k/q}}{\beta_p(k/p)! \cdot \beta_q(k/q)!} \hspace{1.6cm} (\text{Ext. Theorem}) \\
    & < + \infty 
\end{align*}
Consider now the function $f_{(s_1,s_2), (t_1,t_2)}: G\Omega_p(V) \times G\Omega_q(V) \to \overline{T(V)} \times \overline{T(V)}$ defined as follows
\begin{equation}
     f_{(s_1,s_2), (t_1,t_2)}(X, Y)  = \big(S(X)_{s_1,s_2}, S(Y)_{t_1,t_2}\big)
\end{equation}

and the function $g: \overline{T(V)} \times \overline{T(V)} \to \mathbb{R}$ defined as follows
\begin{equation}
    g(T_1, T_2) = \langle T_1, T_2 \rangle
\end{equation}
The map $g$ is clearly continuous in both variables in the sense of $\| \cdot \|$. By the Extension Theorem \ref{thm:extension_theorem} we know that the two maps that extend (uniquely) $X$ and $Y$ respectively to multiplicative functionals on the full tensor algebra $\overline{T(V)}$ are continuous in the $p$- and $q$-variation topologies respectively. Therefore $f_{(s_1,s_2), (t_1,t_2)}$ is also continuous in both of its variables. Hence, $K_{(s_1,s_2), (t_1,t_2)} = g \circ f_{(s_1,s_2), (t_1,t_2)}$ is also continuous in both variables as it is the composition of two continuous functions.
\end{proof}

In the next section we present our second main result. The core technical tool we use in the proof is the notion of \emph{integral of a one-form along a rough path}, discussed in \cref{sec:integration_one_form}.

\subsection{A rough integral equation}
To prove our main result we ought to give a meaning to the following double integral
\begin{equation}\label{def:double_integral_I}
``\mathcal{I}(X, Y) = \int\int K(X, Y) \langle dX, dY\rangle''
\end{equation}
We do so by constructing a double rough integral constructed as the composition of two \emph{one-forms} (\cref{def:Lip_form}) as we shall explain next. In what follows we let $W:=V \oplus \overline{T(V)}$. 

\begin{remark}
    In the following construction, the spaces $V,W$ are swapped compared to the notation used in \cref{sec:integration_one_form}.
\end{remark}

For a fixed tensor $A \in \overline{T(V)}$, consider the linear one-form $\alpha_A: W \to L(W, V)$ defined as follows\footnote{Perhaps more explicitly: for any $(b,B),(b',B') \in W$, $\alpha_A(b,B)(b',B') = \langle A, B \rangle b'$.}: for any $(b,B) \in W$
\begin{equation}\label{eqn:alpha}
\alpha_A(b, B)  = \begin{pmatrix}
                \langle A,B \rangle I_V & 0 \\
                0 & 0 
             \end{pmatrix}
\end{equation}
where $I_V : V \to V$ is the identity on $V$ and where the inner product is taken in $\overline{T(V)}$.
\begin{remark}
    Note that a linear one-form is $Lip(\gamma)$ for all $\gamma \geq 0$ (\cref{def:Lip_form}). Hence, by \cref{def:rough_integral} of \emph{integral of a one-form along a rough path}, we can integrate $\alpha_A$ along any geometric $p$-rough path with $p \geq 1$.
\end{remark}

For any $p \geq 1$ and for any fixed geometric $p$-rough path $Z \in G\Omega_p(V)$, consider now a second linear one-form $\beta_Z : W \to L(W, \mathbb{R})$ defined as follows: for any $(a,A) \in W$ and for any $s,t \in \Delta_I$
\begin{equation}\label{eqn:beta}
    \beta_{Z_{s,t}}(a,A) = \begin{pmatrix}
                \left\langle \left(\int_s^t \alpha_A(Z_u)dZ_u\right)^1, I_V \right\rangle & 0 \\
                0 & 0 
             \end{pmatrix}
\end{equation}
where the inner product is taken in $V$.

\begin{remark}
    Note that the rough integral $\int \alpha_A(Z)dZ$ is a $p$-rough path with values in $T^{\lfloor p \rfloor}(V)$ (and that, by the Extension Theorem \ref{thm:extension_theorem}, its values in $T(V), T((V))$ are also uniquely determined). Here, with the notation $\left(\int \alpha_A(Z_u)dZ_u\right)^1$ we mean the canonical projection of the rough path $\int \alpha_A(Z)dZ$ onto $V$.
\end{remark}

\begin{remark}
    We note that for any $(b,B) \in W$, the data in $b \in V$ is ignored by both one-forms $\alpha_A$ and $\beta_Z$ when acting on $(b,B)$. We preferred to keep this notation as we find it more in line with the standard notation used in rough integration.
\end{remark}

As the one-form $\beta_Z$ is $Lip(\gamma)$ for all $\gamma \geq 0$, we can integrate $\beta_Z$ along any $q$-rough path $\widetilde Z$ with $q \geq 1$ and use this integral as definition for the double integral $\mathcal{I}$ of \cref{def:double_integral_I}.

\begin{definition}\label{def:I} Let $\Delta_I, \Delta_J$ the two simplices and $p,q \geq 1$ be two scalars. Let $X \in G\Omega_p(V)$ and $Y \in G\Omega_q(V)$ be two geometric $p$- and $q$-rough paths respectively. For any $(s_1,s_2) \in \Delta_I$ and any $(t_1,t_2) \in \Delta_J$, define the double rough integral $\mathcal{I}_{(s_1,s_2), (t_1,t_2)}(X,Y)$ as
\begin{equation}\label{eqn:defI}
\mathcal{I}_{(s_1,s_2), (t_1,t_2)}(X,Y) = \left(\int_{u=t_1}^{t_2} \beta_{X_{s_1,s_2}}(Y_u) dY_u\right)^1
\end{equation}
\end{definition}

Note that this definition doesn't depend on the order of integration of $X$ and $Y$. Next is our second main result, an analogue of \cref{thm:PDE} for the case of geometric rough paths.

\begin{theorem}\label{thm:main}
 Let $\Delta_I, \Delta_J$ the two simplices, $p,q \geq 1$ be two scalars, and let $X \in G\Omega_p(V)$ and $Y \in G\Omega_q(V)$ be two geometric $p$- and $q$-rough paths respectively. For any $(s_1,s_2) \in \Delta_I$ and any $(t_1,t_2) \in \Delta_J$ the rough signature kernel of \cref{def:rough_kernel} satisfies the following equation
\begin{equation}\label{eqn:final}
K_{(s_1,s_2), (t_1,t_2)}(X, Y) = 1 + \mathcal{I}_{(s_1,s_2), (t_1,t_2)}(X, Y)
\end{equation}
where $\mathcal{I}$ is the double rough integral of \cref{def:I}.
\end{theorem}

\begin{proof}
By \cite[Theorem 4.12]{lyons2004differential} if $Z \in G\Omega_p(V)$ is a geometric $p$-rough path and $\alpha: V \to L(V,W)$ is a $Lip(\gamma)$ one-form for some $\gamma>p$, then the mapping $Z \mapsto \int \alpha(Z)dZ$ is continuous from $G\Omega_p(V)$ to $G\Omega_p(W)$ in the $p$-variation topology. \medbreak

For any $A \in \overline{T(V)}$ and any $\widetilde Z \in G\Omega_p(V)$  both $\alpha_A$ and $\beta_{\widetilde Z}$ defined in Equations \cref{eqn:alpha} and \cref{eqn:beta} respectively are linear one-forms, hence $Lip(\gamma)$ for any $\gamma \geq 1$. Similarly if $\widetilde Z \in G\Omega_q(V)$. Thus, for any $(s_1,s_2) \in \Delta_I$ and any $(t_1,t_2) \in \Delta_J$, the map $\mathcal{I}_{(s_1,s_2), (t_1,t_2)} : G\Omega_p(V) \times G\Omega_q(V) \to \mathbb{R}$ is continuous in the $p,q$-variation product topology. \medbreak

By \cref{lemma:kernel}, the rough signature kernel $K_{(s_1,s_2), (t_1,t_2)}: G\Omega_p(E) \times G\Omega_q(E) \to \mathbb{R}$ is also continuous with respect to the $p,q$-variation product topology. \medbreak

Following the exact same steps as in the proof of \cref{thm:PDE}, if $X \in G\Omega_1(V)$ and $Y \in G\Omega_1(V)$ are both of bounded variation then the following double integral equation holds 
\begin{equation}
    K_{(s_1,s_2), (t_1,t_2)}(X,Y) = 1 + \int_{s=s_1}^{s_2}\int_{t=t_1}^{t_2} K_{(s_1,s), (t_1,t)}(X,Y)\langle dX_s,dX_t \rangle
\end{equation}
which is equivalent to the equality $K_{(s_1,s_2), (t_1,t_2)}(X, Y) = 1 + \mathcal{I}_{(s_1,s_2), (t_1,t_2)}(X, Y)$. \medbreak

By \cref{def:geom_rough_path} of a geometric $p$- (respectively $q$-) rough path as the limit of $1$-rough paths in the $p$- (respectively $q$-) variation topology, the space of continuous paths of bounded variation $G\Omega_1(V)$ is dense $G\Omega_p(V)$ (respectively $G\Omega_q(V)$). Two continuous functions that are equal on a dense subset of a set are also equal on the whole set. The functional equation $K_{(s_1,s_2), (t_1,t_2)}(\cdot, \cdot) = 1 +  \mathcal{I}_{(s_1,s_2), (t_1,t_2)}(\cdot, \cdot)$ holds on $\Omega^1 G(V) \times \Omega^1 G(V)$, which concludes the proof by the previous density argument.
\end{proof}

This is the last theoretical result of this paper. In \cref{sec:applications} we tackle various machine learning tasks dealing with time series data. 

\section{Data science applications}\label{sec:applications}

In this section we evaluate our signature PDE kernel on three different tasks. Firstly, we consider the task of multivariate time series classification on UEA\footnote{Data available at \url{https://timeseriesclassification.com}} datasets \cite{bagnall2018uea} with a \emph{support vector classifier} (SVC) and compare the performance obtained by equipping the same SVC configuration with various kernel functions, including ours. Secondly, we run a regression task to predict future (average) bitcoin prices from previously observed prices by means of a \emph{support vector regressor} (SVR) and similarly to the previous experiment, we compare the performance produced by a variety of kernels. Lastly, we show how the signature kernel can be easily incorporated within simple optimization procedure to represent the distribution of a large ensemble of paths as a weighted average of a small number of selected paths from the ensemble whilst maintaining certain statistical properties \cite{cosentino2020randomized}. \medbreak

In presence of sequential inputs, well-designed kernels must be chosen with care. \cite{sapankevych2009time}. In the case where all the time series inputs are of the same length, standard kernels on $\mathbb{R}^d$ can be deployed by stacking each dimension of the time series into one single vector. Standard choices of kernels include the linear and Gaussian (a.k.a. RBF) kernels. When the series are not of the same length, other kernels specifically designed for time series can be used to address this issue. Other than the signature PDE kernel introduced in this paper, to our knowledge only two other kernels for sequential data have been proposed in the literature: the truncated signature kernel \cite{kiraly2019kernels} (Sig($n$) - where $n$ denotes the truncation level) and the \emph{global alignment kernel} (GAK) \cite{cuturi2011fast}. For the classification and regression experiments we made use of the SVC and SVR estimators respectively, from the popular python library \texttt{tslearn} \cite{JMLR:v21:20-091}. \medbreak

\paragraph{Hyperparameter selection} The hyperparameters of the SVC and SVR estimators were selected by cross-validation via a grid search on the training set. For the classification we used the train-test split as provided by UEA, whilst for the regression we used a 80-20 split. Both the SVC and SVR estimators depend on a kernel $k$ and on two scalar parameters $C$ and $\gamma$. The range of values for $C$ was chosen to be $\{1, 10, ..., 10^4\}$ and the one for $\gamma$ to be $\{10^{-4},..., 10^4\}$ for all kernel functions included in the comparison. We benchmark our Sig-PDE kernel against the Linear, RBF and GAK \cite{cuturi2011fast} kernels as well as the truncated signature kernel Sig($n$) from \cite{kiraly2019kernels}. We found that the algorithm provided to us by \cite{comm} was in general slower than directly computing the truncated signatures with \texttt{iisignature} \cite{reizenstein2018iisignature}. This is because the former was implemented as pure python code, whilst \texttt{iisignature} is highly optimized and uses a \texttt{C++} backend. For this reason, we ended up computing Sig($n$) without kernel trick for all the experiments. The truncation $n$ is chosen from the range $\{2,...,6\}$. Furthermore, we added a variety of additional hyperparameters to Sig($n$) consisting in: 1) scaling the paths by different scalar scales, 2) normalizing the truncated signatures by multiplying (or not) each level $\ell \in \{1,...,n\}$ by $\ell!$, 3) equipping the SVC/SVR with a Linear or RBF kernel (indexed on truncated signatures). For our Sig-PDE we used the RBF-lifted version with parameter $\sigma$ taken in the range $\{10^{-3}, ..., 10^1\}$. All the experiments are reproducible using the code in \url{https://github.com/crispitagorico/sigkernel} and following the instructions thereafter.

\subsection{Multivariate time series classification}\label{sec:classification}

The support vector classifier (SVC) \cite{vapnik1998support} is one of the simplest yet widely used supervised learning model for classification. It has been successfully used in the fields of text classification \cite{tong2001support}, image retrieval \cite{tong2001support1}, mathematical finance \cite{huang2005forecasting}, medicine \cite{furey2000support} etc. We considered various UEA datasets \cite{bagnall2018uea} of input-output pairs $\{(x_i,y_i)\}_{i=1}^n$ where each $x_i$ is a multivariate time series and each $y_i$ is the corresponding class. In \cref{tab:svc} we display the performance of the same SVC equipped with different kernels (including ours). As the results show, our Sig-PDE kernel is systematically among the top $2$ classifiers across all the datasets (except for FingerMovements and UWaveGestureLibrary) and always outperforms its truncated counterpart, that often overfits during training.

\begin{table}[ht]\small
    \begin{center}
        \begin{tabular}{lccccc}
        \toprule
        \textbf{Datasets/Kernels} &
        Linear &
        RBF &
        GAK & 
        Sig(n) &
        Sig-PDE \\
        \midrule
        ArticularyWordRecognition & 98.0 & 98.0 & 98.0 & 92.3 & \textbf{98.3}\\
        BasicMotions & 87.5 & 97.5 & 97.5 & 97.5 & \textbf{100.0} \\
        Cricket & 91.7 & 91.7 & \textbf{97.2} & 86.1 & \textbf{97.2}\\
        ERing & 92.2 & 92.2 & \textbf{93.7} & 84.1 & 93.3\\
        Libras & 73.9 & 77.2 & 79.0 & \textbf{81.7} & \textbf{81.7}\\
        NATOPS & 90.0 & 92.2 & 90.6 & 88.3 & \textbf{93.3}\\
        RacketSports & 76.9 & 78.3 & 84.2 & 80.2 & \textbf{84.9}\\
        FingerMovements & 57.0 & 60.0 & \textbf{61.0} & 51.0 & 58.0\\
        Heartbeat & 70.2 & 73.2 & 70.2 & 72.2 & \textbf{73.6}\\
        SelfRegulationSCP1 & 86.7 & 87.3 & \textbf{92.4} & 75.4 & 88.7\\
        UWaveGestureLibrary & 80.0 & \textbf{87.5} & \textbf{87.5} & 83.4 & 87.0\\
        \bottomrule
        \end{tabular}
    \end{center}
    \caption{Test set classification accuracy (in \%) on UEA multivariate time series datasets.}
    \label{tab:svc}
\end{table}

\subsection{Predicting Bitcoin prices}
In the last few years, there has a remarkable rise of cryptocurrency trading where the most popular currency, Bitcoin, reached its peak at almost $20,000$ USD/BTC at the end of the year 2017 followed by a big crash in November 2018, when the price dropped to around $3000$ USD/BTC. In this section, we will use daily Bitcoin to USD prices data\footnote{Data is from \url{https://www.cryptodatadownload.com/}} from Gemini which is one of the biggest cryptocurrency trading platforms in the US. Our goal is to use a window of size $36$ days to predict the mean price of the next $2$ days. As shown in \cref{tab:svr}, SVR equipped with the Sig-PDE is able to generalise better to unseen prices and produces better predictions on the test set in terms of MAPE compared to all other benchmarks. We also note that the truncated signature kernel doesn't seem to generalize well to unseen observation for this regression exmaple. In \cref{fig:Bitcoin} we plot the predictions obtained with SVR-Sig-PDE on the train and test sets.

\begin{figure}[h]
    \centering
    \includegraphics[scale=0.5]{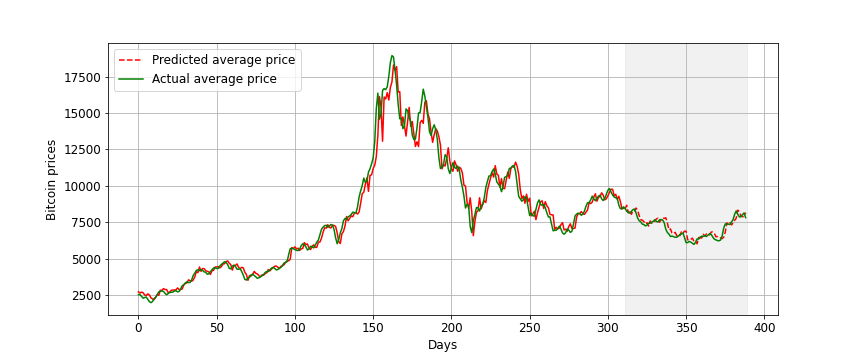}
    \caption{{\small SVR-Sig-PDE predictions of average Bitcoin prices over the next 2 days from Bitcoin prices over the previous 36 days. In white are the predicted prices in the training set; in grey are the predicted prices in the test set. Trading days range from '2017-06-01' to '2018-08-01'.}}
    \label{fig:Bitcoin}
\end{figure}

\begin{table}[ht]\small
    \begin{center}
        \begin{tabular}{lcccc}
        \toprule
        Kernel &
        RBF &
        GAK & 
        Sig(n) &
        Sig-PDE \\
        \midrule
        MAPE & 4.094 & 4.458 & 13.420 & \textbf{3.253} \\
        \bottomrule
        \end{tabular}
    \end{center}
    \caption{Test set \textit{mean absolute percentage error} (MAPE) (in \%) to predict the average Bitcoin price over the next $2$ days given prices over the previous $36$ days.}
    \label{tab:svr}
\end{table}

\subsection{Moments-matching reduction algorithm for  the support of a discrete measure on paths}

As described in \cite{chen2012super},  \emph{ herding} refers to any procedure to approximate integrals of functions in a \emph{reproducing kernel Hilbert space} (RKHS). In particular, such procedure can be useful to estimate \emph{kernel mean embeddings} (KMEs) as we shall explain next. Consider a set $\mathcal{X}$ and a feature map $\Phi$ from $\mathcal{X}$ to an RKHS $H$ with $k$ being the associated positive definite kernel. All elements of $H$ may be identified with real functions $f$ on $\mathcal{X}$ defined by $f(x) = \langle f, \Phi(x) \rangle$ for $x \in \mathcal{X}$. Following \cite{smola2007hilbert} for a fixed probability measure $\mu$ on $\mathcal{X}$ we seek to approximate the KME $\mathbb{E}_\mu \Phi := \int_\mathcal{X} \Phi(x) d\mu(x)$, that belongs to the convex hull of $\{\Phi(x)\}_{x \in \mathcal{X}}$ \cite{bach2012equivalence}. To approximate $\mathbb{E}_\mu \Phi$, we consider $n$ points $x_1, ..., x_n \in \mathcal{X}$ combined linearly with positive weights $w_1, ..., w_n$ that sum to $1$. We then consider the discrete measure $\nu = \sum_{i=1}^n w_i\delta_{x_i}$ and as shown in \cite{bach2012equivalence} we have that
\begin{equation}
    \sup_{f \in H, ||f||\leq 1} |\langle \mathbb{E}_{\nu} \Phi, f\rangle - \langle \mathbb{E}_\mu \Phi, f\rangle| = ||\mathbb{E}_{\nu} \Phi - \mathbb{E}_\mu \Phi||_H
\end{equation}
which means that controlling  $\mathbb{E}_{\hat\mu} \Phi - \mathbb{E}_\mu \Phi$ is enough to control the error
in computing the expectation for all $f \in H$ with norm bounded by $1$. We are interested in the setting where $\mathcal{X}$ is a set of paths of bounded variation taking values on a $d$-dimensional space $E$ (or in practice a set of multivariate time series for example). The signature being a natural feature map for sequential data we set $\Phi = S$, $k$ to be the untruncated signature kernel and $H=\overline{T(E)}$. Following \cite{litterer2012high, cosentino2020randomized}, we consider the problem of reducing the size of the support in $\mathcal{X}$ of a discrete measure $\mu$ whilst preserving some of its statistical properties. Suppose $\#supp(\mu) = N$, where $N$ is large, and $\mu = \sum_{i=1}^N \alpha_i \delta_{x_i}, x_i \in \mathcal{X}$. We call a measure $\nu$ on $\mathcal{X}$ a \emph{reduced measure} with respect to $\mu$ if 
\begin{equation*}
    1) \ supp(\nu) \subset supp(\mu) \quad \text{and} \quad 2) \ 
    \mathbb{E}_\nu S \approx \mathbb{E}_\mu S \quad \text{(in some suitable norm)}
\end{equation*}
We fix the size of the support of the reduced measure $\nu$ to be $\# supp(\nu) = n$, so that $n<<N$. Because of condition 1) we have that $\nu$ is of the form $\mu = \sum_{i=1}^N \beta_i \delta_{x_i}$, where all but $n$ of the weights $\beta_i$'s are equal to $0$. Therefore the vector of weights $\beta=(\beta_1, ..., \beta_N) \in \mathbb{R}^N$ is sparse. We are interested in the following optimization problem
\begin{align*}
\min_{\beta \in \mathbb{R}^N} || \mathbb{E}_\nu S  - \mathbb{E}_\mu S||_{T(E)}^2 &= \min_{\beta \in \mathbb{R}^N} || \sum_{i=1}^N (\alpha_i - \beta_i)S(x_i)||_{T(E)}^2 \\
&= \min_{\beta \in \mathbb{R}^N} \Big\langle \sum_{i=1}^N (\alpha_i - \beta_i)S(x_i), \sum_{j=1}^N (\alpha_j - \beta_j)S(x_i) \Big\rangle_{T(E)} \\
&= \min_{\beta \in \mathbb{R}^N} \underbrace{\sum_{i,j=1}^N (\alpha_i - \beta_i)(\alpha_j - \beta_j) k(x_i,x_j)}_{:= L(\beta)}
\end{align*}

where $k$ is the signature kernel. This minimisation will not yield a sparse vector $\beta$. To induce sparsity we use an $l_1$ penalisation on the weights $\beta$ as in LASSO, which amounts to the following Lagrangian minimisation
\begin{equation}\label{eqn:min_lasso}
\min_{\beta \in \mathbb{R}^N} L(\beta) + \lambda ||\beta||_1
\end{equation}
where $\lambda$ is a penalty parameter determined by the size $n$ of the support of $\nu$. Equation (\ref{eqn:min_lasso}) minimises a function $f:\mathbb{R}^N \to \mathbb{R}$ that can be decomposed as $f = L + h$, where $L$ is differentiable and $h=\lambda ||\cdot||_1$ is convex but non-differentiable, so a gradient descent algorithm can't be directly applied. \textit{Subgradient descent methods} are classical algorithms that address this issue but have poor convergence rates \cite{bach2011optimization}. A better choice of algorithms for this particular problem are called \textit{proximal gradient methods} \cite{schmidt2011convergence}. Define the soft-thresholding operator $A_{\gamma} : \mathbb{R}^N \to \mathbb{R}^N$ as follows
\begin{equation}
A_{\gamma}(\beta)_i =
\begin{cases}
  \beta_i - \gamma, & \text{if $\beta_i>\gamma$}\\
  0, & \text{if $|\beta_i|\leq\gamma$}\\
  \beta_i + \gamma, & \text{if $\beta_i<-\gamma$}
\end{cases}
\end{equation}

Then, it can be shown \cite{schmidt2011convergence} that $\beta^*$ is a minimiser of the optimisation (\ref{eqn:min_lasso}) if and only if $\beta^*$ solves the following fixed point problem
\begin{equation}\label{eqn:fixed_point}
    \beta^* = A_\gamma(\beta^* - \gamma \nabla_\beta L(\beta^*))
\end{equation}

The fixed point problem \cref{eqn:fixed_point} can be solved iteratively fixing $\beta^0 \in \mathbb{R}^N$ and for $k\geq 1$
\begin{equation}
    \beta^{k+1} = A_\gamma(\beta^k - \gamma \nabla_\beta L(\beta^k))
\end{equation}

Proximal gradient descent methods convergence with rate $O(1/\epsilon)$ which is an order of magnitude better that the $O(1/\epsilon^2)$ convergence rate of subgradient methods \cite{schmidt2011convergence}. 

\begin{figure}[ht]
\centering
\centerline{\includegraphics[scale=0.45]{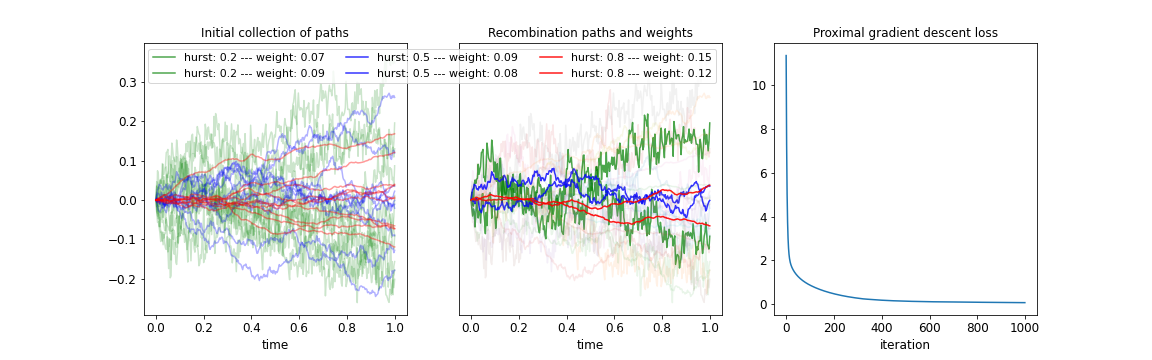}}
\caption{{\small On the \textbf{left} are $30$ fBM sample paths. In the \textbf{middle} are the results obtained by solving the optimisation \cref{eqn:min_lasso}. On the \textbf{right} is the loss as a function of the proximal gradient descent iteration.}}
\label{fig:recombination}
\end{figure}

We apply the above proximal gradient descent algorithm to an example of a set of $30$ sample paths of fractional Brownian Motion (fBM) with Hurst exponent drawn uniformly at random from $\{0.2,0.5,0.8\}$ (note that fBM($0.5$) corresponds to Brownian motion). The goal is to compute a reduced measure with smaller support size. We choose a value of the penalisation constant $\lambda$ in \cref{eqn:min_lasso} so that the new support is of size $=6$. The selected paths with corresponding weights are displayed in \cref{fig:recombination}. This selection is clearly well-balanced across the samples ($2$ paths per exponent) so more likely to well-represent the initial ensemble.

\begin{remark}
    We note that our signature PDE kernel has been succesfully deployed to perform \emph{distribution regression on sequential data} in \cite{lemercier2020distribution} and malware detection in \cite{cochrane2021sk}.
\end{remark}

\section{Conclusion}

In this paper we introduce the signature PDE kernel and show that when paths are continuously differentiable the latter solves a hyperbolic PDE. We recognize the connection with a well known class of differential equations known in the literature as Goursat problems. Our Goursat PDE can be solved numerically using state-of-the-art hyperbolic PDE solvers; we propose an efficient finite different scheme to do so that has linear time complexity when implemented on \texttt{GPU} and analyse its convergence properties. We extend the previous analysis to the case of geometric rough paths and establish a rough integral equation analogous to the aforementioned Goursat problem. Finally we demonstrated the effectiveness of our kernel in various data science applications dealing sequential data. 

\section*{Acknowledgements}

We thank Maud Lemercier for her help with the implementation of \texttt{sigkernel}, Franz Kiraly and Harald Oberhauser for the helpful comments and the referees for pointing out an error in the experiments in the previous version of the paper.

\bibliographystyle{unsrt}
\bibliography{references}

\begin{thebibliography}{10}

\bibitem{kiraly2019kernels}
Franz~J Kir{\'a}ly and Harald Oberhauser.
\newblock Kernels for sequentially ordered data.
\newblock {\em Journal of Machine Learning Research}, 2019.

\bibitem{lyons2004differential}
Terry Lyons, Michael Caruana, and Thierry L{\'e}vy.
\newblock Differential equations driven by rough paths.
\newblock {\em Ecole d’{\'e}t{\'e} de Probabilit{\'e}s de Saint-Flour XXXIV},
  pages 1--93, 2004.

\bibitem{lyons2014roughICM}
Terry Lyons.
\newblock Rough paths, signatures and the modelling of functions on streams.
\newblock {\em International Congress of Mathematicians, Seoul}, 2014.

\bibitem{arribas2018signature}
Imanol~Perez Arribas, Guy~M Goodwin, John~R Geddes, Terry Lyons, and Kate~EA
  Saunders.
\newblock A signature-based machine learning model for distinguishing bipolar
  disorder and borderline personality disorder.
\newblock {\em Translational psychiatry}, 8(1):1--7, 2018.

\bibitem{morrill2020utilization}
James~H Morrill, Andrey Kormilitzin, Alejo~J Nevado-Holgado, Sumanth
  Swaminathan, Samuel~D Howison, and Terry~J Lyons.
\newblock Utilization of the signature method to identify the early onset of
  sepsis from multivariate physiological time series in critical care
  monitoring.
\newblock {\em Critical Care Medicine}, 48(10):e976--e981, 2020.

\bibitem{bonnier2019deep}
Patrick Kidger, Patric Bonnier, Imanol Perez~Arribas, Cristopher Salvi, and
  Terry Lyons.
\newblock Deep signature transforms.
\newblock In {\em Advances in Neural Information Processing Systems},
  volume~32, 2019.

\bibitem{yang2017developing}
Weixin Yang, Terry Lyons, Hao Ni, Cordelia Schmid, and Lianwen Jin.
\newblock Developing the path signature methodology and its application to
  landmark-based human action recognition.
\newblock {\em arXiv preprint arXiv:1707.03993}, 2017.

\bibitem{moore2019using}
PJ~Moore, TJ~Lyons, J~Gallacher, and Alzheimer’s Disease~Neuroimaging
  Initiative.
\newblock Using path signatures to predict a diagnosis of alzheimer’s
  disease.
\newblock {\em PloS one}, 14(9):e0222212, 2019.

\bibitem{cochrane2021sk}
Thomas Cochrane, Peter Foster, Varun Chhabra, Maud Lemercier, Cristopher Salvi,
  and Terry Lyons.
\newblock Sk-tree: a systematic malware detection algorithm on streaming trees
  via the signature kernel.
\newblock {\em arXiv preprint arXiv:2102.07904}, 2021.

\bibitem{lyons2019numerical}
Terry Lyons, Sina Nejad, and Imanol Perez~Arribas.
\newblock Numerical method for model-free pricing of exotic derivatives in
  discrete time using rough path signatures.
\newblock {\em Applied Mathematical Finance}, 26(6):583--597, 2019.

\bibitem{hofmann2008kernel}
Thomas Hofmann, Bernhard Sch{\"o}lkopf, and Alexander~J Smola.
\newblock Kernel methods in machine learning.
\newblock {\em The annals of statistics}, pages 1171--1220, 2008.

\bibitem{shawe2004kernel}
John Shawe-Taylor, Nello Cristianini, et~al.
\newblock {\em Kernel methods for pattern analysis}.
\newblock Cambridge university press, 2004.

\bibitem{cuturi2011fast}
Marco Cuturi.
\newblock Fast global alignment kernels.
\newblock In {\em Proceedings of the 28th international conference on machine
  learning (ICML-11)}, pages 929--936, 2011.

\bibitem{goursat1916course}
Edouard Goursat.
\newblock {\em A Course in Mathematical Analysis: pt. 2. Differential
  equations.[c1917}, volume~2.
\newblock Dover Publications, 1916.

\bibitem{Tth2019VariationalGP}
Csaba Toth and Harald Oberhauser.
\newblock Bayesian learning from sequential data using gaussian processes with
  signature covariances.
\newblock In {\em Proceedings of the international conference on machine
  learning (ICML)}, 2020.

\bibitem{chevyrev2018signature}
Ilya Chevyrev and Harald Oberhauser.
\newblock Signature moments to characterize laws of stochastic processes.
\newblock {\em arXiv preprint arXiv:1810.10971}, 2018.

\bibitem{chouk2014rough}
Khalil Chouk and Massimiliano Gubinelli.
\newblock Rough sheets.
\newblock {\em arXiv preprint arXiv:1406.7748}, 2014.

\bibitem{kiraly2016kernels}
Franz~J Kir{\'a}ly and Harald Oberhauser.
\newblock Kernels for sequentially ordered data.
\newblock {\em arXiv preprint arXiv:1601.08169}, 2016.

\bibitem{lees1960goursat}
Milton Lees.
\newblock The goursat problem.
\newblock {\em Journal of the Society for Industrial and Applied Mathematics},
  8(3):518--530, 1960.

\bibitem{day1966finitediff}
J.~T. Day.
\newblock A runge-kutta method for the numerical solution of the goursat
  problem in hyperbolic partial differential equations.
\newblock {\em The Computer Journal}, 9:81--83, 1966.

\bibitem{wazwaz1993finitediff}
A.~M. Wazwaz.
\newblock On the numerical solution for the goursat problem.
\newblock {\em Applied Mathematics and Computation}, 59:89--95, 1993.

\bibitem{tzen2019neural}
Belinda Tzen and Maxim Raginsky.
\newblock Neural stochastic differential equations: Deep latent gaussian models
  in the diffusion limit.
\newblock {\em arXiv preprint arXiv:1905.09883}, 2019.

\bibitem{kidger2021neural}
Patrick Kidger, James Foster, Xuechen Li, Harald Oberhauser, and Terry Lyons.
\newblock Neural sdes as infinite-dimensional gans.
\newblock {\em arXiv preprint arXiv:2102.03657}, 2021.

\bibitem{li2020scalable}
Xuechen Li, Ting-Kam~Leonard Wong, Ricky~TQ Chen, and David~K Duvenaud.
\newblock Scalable gradients and variational inference for stochastic
  differential equations.
\newblock In {\em Symposium on Advances in Approximate Bayesian Inference},
  pages 1--28. PMLR, 2020.

\bibitem{cuchiero2020generative}
Christa Cuchiero, Wahid Khosrawi, and Josef Teichmann.
\newblock A generative adversarial network approach to calibration of local
  stochastic volatility models.
\newblock {\em Risks}, 8(4):101, 2020.

\bibitem{gatheral2018volatility}
Jim Gatheral, Thibault Jaisson, and Mathieu Rosenbaum.
\newblock Volatility is rough.
\newblock {\em Quantitative Finance}, 18(6):933--949, 2018.

\bibitem{bayer2016pricing}
Christian Bayer, Peter Friz, and Jim Gatheral.
\newblock Pricing under rough volatility.
\newblock {\em Quantitative Finance}, 16(6):887--904, 2016.

\bibitem{bagnall2018uea}
A.~Bagnall, J.~Lines, A.~Bostrom, J.~Large, and E.~Keogh.
\newblock The great time series classification bake off: a review and
  experimental evaluation of recent algorithmic advances.
\newblock {\em Data Mining and Knowledge Discovery}, 31:606--660, 2017.

\bibitem{cosentino2020randomized}
Francesco Cosentino, Harald Oberhauser, and Alessandro Abate.
\newblock A randomized algorithm to reduce the support of discrete measures.
\newblock {\em arXiv preprint arXiv:2006.01757}, 2020.

\bibitem{sapankevych2009time}
Nicholas~I Sapankevych and Ravi Sankar.
\newblock Time series prediction using support vector machines: a survey.
\newblock {\em IEEE Computational Intelligence Magazine}, 4(2):24--38, 2009.

\bibitem{JMLR:v21:20-091}
Romain Tavenard, Johann Faouzi, Gilles Vandewiele, Felix Divo, Guillaume
  Androz, Chester Holtz, Marie Payne, Roman Yurchak, Marc Ru{\ss}wurm, Kushal
  Kolar, and Eli Woods.
\newblock Tslearn, a machine learning toolkit for time series data.
\newblock {\em Journal of Machine Learning Research}, 21(118):1--6, 2020.

\bibitem{comm}
C.~Toth C.~Salvi.
\newblock personal communication.

\bibitem{reizenstein2018iisignature}
Jeremy Reizenstein and Benjamin Graham.
\newblock The iisignature library: efficient calculation of iterated-integral
  signatures and log signatures.
\newblock {\em arXiv preprint arXiv:1802.08252}, 2018.

\bibitem{vapnik1998support}
Vladimir Vapnik.
\newblock The support vector method of function estimation.
\newblock In {\em Nonlinear Modeling}, pages 55--85. Springer, 1998.

\bibitem{tong2001support}
Simon Tong and Daphne Koller.
\newblock Support vector machine active learning with applications to text
  classification.
\newblock {\em Journal of machine learning research}, 2(Nov):45--66, 2001.

\bibitem{tong2001support1}
Simon Tong and Edward Chang.
\newblock Support vector machine active learning for image retrieval.
\newblock In {\em Proceedings of the ninth ACM international conference on
  Multimedia}, pages 107--118, 2001.

\bibitem{huang2005forecasting}
Wei Huang, Yoshiteru Nakamori, and Shou-Yang Wang.
\newblock Forecasting stock market movement direction with support vector
  machine.
\newblock {\em Computers \& operations research}, 32(10):2513--2522, 2005.

\bibitem{furey2000support}
Terrence~S Furey, Nello Cristianini, Nigel Duffy, David~W Bednarski, Michel
  Schummer, and David Haussler.
\newblock Support vector machine classification and validation of cancer tissue
  samples using microarray expression data.
\newblock {\em Bioinformatics}, 16(10):906--914, 2000.

\bibitem{chen2012super}
Yutian Chen, Max Welling, and Alex Smola.
\newblock Super-samples from kernel herding.
\newblock In {\em Proceedings of the Twenty-Sixth Conference on Uncertainty in
  Artificial Intelligence}, pages 109--116, 2010.

\bibitem{smola2007hilbert}
Alex Smola, Arthur Gretton, Le~Song, and Bernhard Sch{\"o}lkopf.
\newblock A hilbert space embedding for distributions.
\newblock In {\em International Conference on Algorithmic Learning Theory},
  pages 13--31. Springer, 2007.

\bibitem{bach2012equivalence}
Francis~R Bach, Simon Lacoste-Julien, and Guillaume Obozinski.
\newblock On the equivalence between herding and conditional gradient
  algorithms.
\newblock In {\em ICML}, 2012.

\bibitem{litterer2012high}
Christian Litterer, Terry Lyons, et~al.
\newblock High order recombination and an application to cubature on wiener
  space.
\newblock {\em The Annals of Applied Probability}, 22(4):1301--1327, 2012.

\bibitem{bach2011optimization}
Francis Bach, Rodolphe Jenatton, Julien Mairal, and Guillaume Obozinski.
\newblock Optimization with sparsity-inducing penalties.
\newblock {\em Found. Trends Mach. Learn.}, 2012.

\bibitem{schmidt2011convergence}
Mark Schmidt, Nicolas~L Roux, and Francis~R Bach.
\newblock Convergence rates of inexact proximal-gradient methods for convex
  optimization.
\newblock In {\em Advances in neural information processing systems}, pages
  1458--1466, 2011.

\bibitem{lemercier2020distribution}
Maud Lemercier, Cristopher Salvi, Theodoros Damoulas, Edwin~V Bonilla, and
  Terry Lyons.
\newblock Distribution regression for continuous-time processes via the
  expected signature.
\newblock {\em arXiv preprint arXiv:2006.05805}, 2020.

\bibitem{polyanin2015handbook}
Andrei~D. Polyanin and Vladimir~E. Nazaikinskii.
\newblock {\em Handbook of Linear Partial Differential Equations for Engineers
  and Scientists}.
\newblock 2nd Edition, Chapman and Hall/CRC, 2015.

\bibitem{andrews2001functions}
George~E. Andrews, Richard Askey, and Ranjan Roy.
\newblock Special functions.
\newblock {\em Encyclopedia of Mathematics and its Applications}, 71, 2001.

\bibitem{lyons1998differential}
Terry~J Lyons.
\newblock Differential equations driven by rough signals.
\newblock {\em Revista Matem{\'a}tica Iberoamericana}, 14(2):215--310, 1998.

\bibitem{esig}
Terry~Lyons et~al.
\newblock Coropa computational rough paths (software library).
\newblock 2010.

\bibitem{signatory}
Patrick Kidger and Terry Lyons.
\newblock {Signatory: differentiable computations of the signature and
  logsignature transforms, on both CPU and GPU}.
\newblock {\em arXiv:2001.00706}, 2020.

\bibitem{lyons2007extension}
Terry Lyons and Nicolas Victoir.
\newblock An extension theorem to rough paths.
\newblock In {\em Annales de l'IHP Analyse non lin{\'e}aire}, volume~24, pages
  835--847, 2007.

\end{thebibliography}

\newpage

\appendix

\section{Error analysis for the numerical scheme}\label{append:error} In this section, we show that the finite difference scheme \cref{eqn:explicit_finite_diff} achieves a second order convergence rate for the Goursat problem \cref{eqn:PDE}. Our analysis is based on an explicit representation of the PDE solution.\medbreak

\begin{theorem}[Example 17.4 from \cite{polyanin2015handbook}]\label{thm:simple_Goursat_sol}
Consider the following specific case of the general Goursat problem (\ref{eqn:Goursat}) on the domain $\mathcal{D} = \{(s,t) \mid u\leq s \leq u', v \leq t \leq v'\}$:
\begin{equation}\label{eqn:simple_Goursat}
\frac{\partial^2 k}{\partial s \partial t} = C_3 k,
\end{equation}
where $C_3$ is constant and the boundary data $k(s,v) = \sigma(s)$, $k(u,t) = \tau(t)$ is differentiable. Then the solution $k$ can be expressed as
\begin{equation}\label{eqn:simple_Goursat_sol1}
k(s,t) = k(u,v)\,R(s-u,t-v) + \int_u^s \sigma^{\prime}(r)\,R(s - r,t - v)\,dr + \int_v^t \tau^{\prime}(r)\,R(s - u, t - r)\,dr,
\end{equation}
for $s,t\in\mathcal{D}$, where the Riemann function $R$ is defined as $R(a,b) := J_0\big(2i\sqrt{C_{3\,} ab}\,\big)$ for $a,b\geq 0$, with $J_0$ denoting the zero order Bessel function of the first kind.
\end{theorem}\medbreak

\begin{remark}
To simplify notation, we shall use the $n$-th order modified Bessel function $I_n(z) := i^{-n}J_n(iz)$. It directly follows from the series expansion of $J_n(2z)$ \cite[Section 4.5]{andrews2001functions} that 
\begin{equation}\label{eqn:bessel_series}
I_n(2z) = \Bigg(\sum_{k=0}^\infty \frac{z^{2k}}{k!(n+k)!}\Bigg)z^n,
\end{equation}
From the identities (4.6.1), (4.6.2), (4.6.5), (4.6.6) in \cite{andrews2001functions}, we can compute derivatives of $I_0$ as
\begin{align}
I_0^\prime(z) & = I_1(z),\label{eqn:bessel_identity_1} \\[3pt]
I_0^{\prime\prime}(z) & = I_2(z) + z^{-1}I_1(z).\label{eqn:bessel_identity_2}
\end{align}
\end{remark}\medbreak

Using \cref{thm:simple_Goursat_sol} and the above identities, we will perform a local error analysis for the explicit scheme \cref{eqn:explicit_finite_diff}. \medbreak

\begin{theorem}[Local error estimates for the explicit scheme]\label{thm:local_error} Consider the Goursat problem \cref{eqn:simple_Goursat} on the domain $\mathcal{D} = \{(s,t) \mid u\leq s \leq u', v \leq t \leq v'\}$:
\begin{equation*}
\frac{\partial^2 k}{\partial s \partial t} = C_3 k,
\end{equation*}
where $C_3$ is constant and the boundary data $u(s,v) = \sigma(s)$, $u(u,t) = \tau(t)$ is differentiable
and of bounded variation. We define the local approximation error of the explicit scheme \cref{eqn:explicit_finite_diff} as
\begin{equation*}
E(s,t) := k(s,t) - \Big(k(s,v) + k(u,t) - k(u,v) + \frac{1}{2}\big(k(s,v) + k(u,t)\big)\, C_3(s-u)(t-v)\Big).
\end{equation*}
Then
\begin{align}\label{eqn:error_estimate_1}
|E(s,t)| & \leq \frac{1}{2}|C_3|\big(\|\sigma\|_{1,[u,s]} + \|\tau\|_{1,[v,t]}\big)(s-u)(t-v) \sup_{z\in[0,C_3(s-u)(t-v)]}\bigg|\frac{I_1(2\sqrt{z}\,)}{\sqrt{r}}\bigg|\\[3pt]
&\hspace{8.5mm} + \frac{1}{2}|C_3||k(s,v) + k(u,t)|(s-u)(t-v) \sup_{z\in[0,C_{3}(s-u)(t-v)]}\bigg|\frac{I_1(2\sqrt{z}\,)}{\sqrt{z}} - 1\bigg|\,.\nonumber
\end{align}
In addition, if $\sigma, \tau$ are twice differentiable and their derivatives have bounded variation, then
\begin{align}\label{eqn:error_estimate_2}
|E(s,t)| & \leq \frac{1}{2}|C_3|\big(\|\sigma^\prime\|_{1,[u,s]}(s-u) + \|\tau^\prime\|_{1,[v,t]}(t-v)\big)(s-u)(t-v)\sup_{z\in[0,C_3(s-u)(t-v)]}\bigg|\frac{I_1(2\sqrt{z}\,)}{\sqrt{z}}\bigg|\\[3pt]
&\hspace{3.5mm} + \frac{1}{12}|C_3|^2\big(|\sigma^\prime(u)|(s-u) + |\tau^\prime(v)|(t-v)\big)(s-u)^2(t-v)^2\hspace{-1mm}\sup_{z\in[0,C_{3}(s-u)(t-v)]}\bigg|\frac{I_2(2\sqrt{z}\,)}{z}\bigg|\nonumber\\[3pt]
&\hspace{3.5mm} + \frac{1}{2}|C_3||k(s,v) + k(u,t)|(s-u)(t-v)\sup_{z\in[0,C_{3}(s-u)(t-v)]}\bigg|\frac{I_1(2\sqrt{z}\,)}{\sqrt{r}} - 1\bigg|\,.\nonumber
\end{align}
\end{theorem}
\begin{remark}
From \cref{eqn:bessel_series}, it is clear that $\frac{I_1(2\sqrt{z}\,)}{\sqrt{z}}\sim 1$, $\frac{I_2(2\sqrt{z}\,)}{z}\sim \frac{1}{2}$ and $\frac{I_1(2\sqrt{z}\,)}{\sqrt{z}} - 1\sim \frac{1}{2}z$.
\end{remark}\medbreak
\begin{proof}
To begin, we decompose the approximation error as $E(s,t) = E_1 + E_2$ where
\begin{align*}
E_1 & := k(s,t) - \Big(k(s,v) + k(u,t) - k(u,v) + \frac{1}{2}\big(k(s,v) + k(u,t)\big)\big(R(s-u,t-v) - 1\big)\Big),\\[3pt]
E_2 & := \frac{1}{2}\big(k(s,v) + k(u,t)\big)\big(C_3(s-u)(t-v) - \big(R(s-u,t-v) - 1\big)\big). 
\end{align*}
Since $R(s - u, 0) = R(0, t-v) = 1$, $\sigma(s) = \tau(v) = k(u,v)$, $\sigma(s) =k(s,v)$ and $\tau(t) =  k(u,t)$, it follows that
\begin{align*}
k(s,v) + k(u,t) - k(u,v) + \frac{1}{2}& \big(k(s,v) + k(u,t)\big)\big(R(s-u,t-v) - 1\big)\\
= k(u,v)\,R(s-u,t-v) & + \frac{1}{2}\big(k(s,v) -2k(u,v)+ k(u,t)\big)\big(R(s-u,t-v) + 1\big)\\[3pt]
= k(u,v)\,R(s-u,t-v) & + \frac{1}{2}\big((\sigma(s) - \sigma(u)) + (\tau(t) - \tau(v))\big)\big(R(s - u,t - v) + 1\big)\\[3pt]
= k(u,v)\,R(s-u,t-v) & + \frac{1}{2}\bigg(\int_u^s \sigma^{\prime}(r)\,dr + \int_v^t \tau^{\prime}(r)\,dr \bigg)\big(R(s - u,t - v) + 1\big)\\[3pt]
= k(u,v)\,R(s-u,t-v) & + \frac{1}{2}\int_u^s \sigma^{\prime}(r)\big(R(0,t - v) + R(s - u,t - v)\big)\,dr\\
& + \frac{1}{2}\int_v^t \tau^{\prime}(r)\big(R(s - u, 0) + R(s - u,t - v)\big)\,dr.
\end{align*}
Hence by (\ref{eqn:simple_Goursat_sol1}), we can write $E_1 = E_3 + E_4$ where the error terms $E_3$ and $E_4$ are given by
\begin{align*}
E_3 & := \int_u^s \sigma^{\prime}(r)\bigg(R(s - r,t - v) - \frac{1}{2}\big(R(0,t - v) + R(s - u,t - v)\big)\bigg)\,dr,\\[3pt]
E_4 & := \int_v^t \tau^{\prime}(r)\bigg(R(s - u,t - r) - \frac{1}{2}\big(R(s - u, 0) + R(s - u,t - v)\big)\bigg)\,dr.
\end{align*}
The integrand of $E_3$ can be estimated as
\begin{align}
&\bigg|R(s - r,t - v) - \frac{1}{2}\big(R(0,t - v) + R(s - u,t - v)\big)\bigg|\label{eqn:basic_r_esimate}\\
&\hspace{5mm} = \bigg|\frac{1}{2}\big(R(s - r,t - v) - R(0,t - v)\big) - \frac{1}{2}\big(R(s - u,t - v) - R(s - r,t - v)\big)\bigg|\nonumber\\[3pt]
&\hspace{5mm} \leq \frac{1}{2}\bigg|\int_0^{s-r} \frac{\partial R}{\partial w}(w,t - v)\,dw \bigg| + \frac{1}{2}\bigg|\int_{s-r}^{s-u} \frac{\partial R}{\partial w}(w,t - v)\,dw \bigg|\nonumber\\[3pt]
&\hspace{5mm} \leq \frac{1}{2}(s-r)\sup_{w\in[0,s-r]}\bigg|\frac{\partial R}{\partial w}(w,t - v)\bigg| + \frac{1}{2}(r-u)\sup_{w\in[s-r,s-u]}\bigg|\frac{\partial R}{\partial w}(w,t - v)\bigg|\nonumber\\[3pt]
&\hspace{5mm}\leq \frac{1}{2}(s-u)\sup_{w\in[0,s-u]}\bigg|\frac{\partial R}{\partial w}(w,t - v)\bigg|.\nonumber
\end{align}
By applying the formulae \cref{eqn:bessel_identity_1} and \cref{eqn:bessel_identity_2} to the function $R$, we can compute its derivatives,
\begin{align}
\frac{\partial R}{\partial w}(w,t - v) = \frac{C_3(t-v)I_1\big(2\sqrt{C_3(w(t-v)}\,\big)}{\sqrt{C_3 w(t-v)}}\,,\label{eqn:R_partial_1}\\[3pt]
\frac{\partial^2 R}{\partial^2 w}(w,t - v) = \frac{C_3(t-v) I_2\big(2\sqrt{C_3(w(t-v)}\,\big)}{w}\,.\label{eqn:R_partial_2}
\end{align}
Therefore, it now follows from \cref{eqn:basic_r_esimate} and \cref{eqn:R_partial_1} that
\begin{align*}
|E_3| & \leq \int_u^s \big|\sigma^{\prime}(r)\big|\bigg|R(s - r,t - v) - \frac{1}{2}\big(R(0,t - v) + R(s - u,t - v)\big)\bigg|\,dr\\
& \leq \frac{1}{2}\int_u^s \big|\sigma^{\prime}(r)\big|\,dr\,(s-u)\sup_{w\in[0,s-u]}\bigg|\frac{\partial R}{\partial w}(w,t - v)\bigg|\\
& \leq \frac{1}{2}|C_3|\|\sigma\|_{1,[u,s]}(s-u)(t-v)\sup_{z\in[0,C_3(s-u)(t-v)]}\bigg|\frac{I_1(2\sqrt{z}\,)}{\sqrt{z}}\bigg|\,,
\end{align*}
and we can obtain a similar estimate for $E_4$ (where $\|\tau\|_{1,[v,t]}$ would appear instead of $\|\sigma\|_{1,[u,s]}$). From the estimates for $E_3$ and $E_4$, we have
\begin{equation*}
|E_1| \leq \frac{1}{2}|C_3|\big(\|\sigma\|_{1,[u,s]} + \|\tau\|_{1,[v,t]}\big)(s-u)(t-v)\sup_{z\in[0,C_3(s-u)(t-v)]}\bigg|\frac{I_1(2\sqrt{z}\,)}{\sqrt{z}}\bigg|\,.
\end{equation*}
Estimating $E_2$ is straightforward as
\begin{align*}
&\big|R(s - u,t - v) - \big(1 + C_{3}(s-u)(t-v)\big)\big|\\[3pt]
&\hspace{5mm} = \big|\big(I_0\big(2\sqrt{C_{3\,}(s-u)(t-v)}\,\big) - I_0(0)\big) - C_{3}(s-u)(t-v)\big|\\[3pt]
&\hspace{5mm} = \bigg|\int_0^{C_{3}(s-u)(t-v)}\bigg(\frac{I_1(2\sqrt{z}\,)}{\sqrt{z}} - 1\bigg)\,dz\bigg|\\[3pt]
&\hspace{5mm} \leq |C_{3}|(s-u)(t-v)\sup_{z\in[0,C_3(s-u)(t-v)]}\bigg|\frac{I_1(2\sqrt{z}\,)}{\sqrt{z}} - 1\bigg|\,.
\end{align*}
Using the above estimates for $E_1$ and $E_2$, we obtain \cref{eqn:error_estimate_1} as required. For the remainder of this proof we will assume that $\sigma$, $\tau$ are twice differentiable and $\sigma^\prime$, $\tau^\prime$ have bounded variation.
In this case, we can apply the fundamental theorem of calculus to the integrand of $E_3$ so that
\begin{align*}
E_3 & = \int_u^s \sigma^{\prime}(r)\bigg(R(s - r,t - v) - \frac{1}{2}\big(R(0,t - v) + R(s - u,t - v)\big)\bigg)\,dr\\
& = \int_u^s \bigg(\sigma^{\prime}(u) + \int_u^r \sigma^{\prime\prime}(w)\,dw\bigg)\bigg(R(s - r,t - v) - \frac{1}{2}\big(R(0,t - v) + R(s - u,t - v)\big)\bigg)\,dr.
\end{align*}
Note that by the well-known error estimate for the Trapezium rule, we have
\begin{align*}
&\bigg|\int_u^s R(s - r,t - v)\,dr - \frac{1}{2}(s-u)\big(R(0,t - v) + R(s - u,t - v)\big)\bigg|
&\\
&\hspace{5mm} \leq \frac{1}{12}(s-u)^3\sup_{w\in[0,s-u]}\bigg|\frac{\partial^2 R}{\partial w^2}(w,t - v)\bigg|\,.
\end{align*}
Recall that this derivative was given by \cref{eqn:R_partial_2}. It now follows from the above and \cref{eqn:basic_r_esimate} that
\begin{align*}
|E_3| & \leq \big|\sigma^{\prime}(u)\big|\,\bigg|\int_u^s R(s - r,t - v)\,dr - \frac{1}{2}(s-u)\big(R(0,t - v) + R(s - u,t - v)\big)\bigg|\\
&\hspace{10mm} + \int_u^s \int_u^r \big|\sigma^{\prime\prime}(w)\big|\,dw\,\bigg|R(s - r,t - v) - \frac{1}{2}\big(R(0,t - v) + R(s - u,t - v)\big)\bigg|\,dr\\[3pt]
&\leq \frac{1}{12}|C_3|^2|\sigma^{\prime}(u)|(s-u)^3(t-v)^2\sup_{z\in[0,C_{3}(s-u)(t-v)]}\bigg|\frac{I_2(2\sqrt{z}\,)}{z}\bigg|\\
&\hspace{10mm} + \frac{1}{2}|C_3|\|\sigma^\prime\|_{1,[u,s]}(s-u)^2(t-v)\sup_{z\in[0,C_3(s-u)(t-v)]}\bigg|\frac{I_1(2\sqrt{z}\,)}{\sqrt{z}}\bigg|\,.
\end{align*}
Applying the same argument to $E_4$ leads to the second estimate \cref{eqn:error_estimate_2} as required.
\end{proof}\medbreak



From the estimate \cref{eqn:error_estimate_2} we see that the proposed finite difference scheme achieves a local error that is $O(h^4)$ when the domain $\mathcal{D}$ has a small height and width of $h$ (and provided the boundary data is smooth enough). Since discretizing a PDE on an $n\times n$ grid involves $(n-1)^2$ steps, we expect the proposed scheme to have a second order of convergence.\medbreak

\begin{theorem}[Global error estimate]
Let $\widetilde{k}$ be a numerical solution obtained by applying the proposed finite difference scheme (\cref{def:fintite_diff}) to the Goursat PDE \cref{eqn:PDE} on the grid $P_\lambda$ where $x$ and $y$ are piecewise linear with respect to the grids $\mathcal{D}_I$ and $\mathcal{D}_J$. In particular, we are assuming there exists a constant $M$, that is independent of $\lambda$, such that
\begin{equation*}
\sup_{\mathcal{D}}|\langle \dot x_s, \dot y_t \rangle| < M.
\end{equation*}
Then there exists a constant $K > 0$ depending on $M$ and $k_{x,y}$, but independent of $\lambda$, such that
\begin{equation}
\sup_{\mathcal{D}}\big|k_{x,y}(s, t) - \widetilde{k}(s, t)\big| \leq \frac{K}{2^{2\lambda}}\,,
\end{equation}
for all $\lambda \geq 0$.
\end{theorem}\medbreak
\begin{proof}
Using the solution $k_{x,y}$, we define another approximation $k^{\prime}$ on $P_\lambda$ as  
\begin{align*}
k^{\prime}(s_{i+1},t_{j+1}) & := k_{x,y}(s_{i+1}, t_j) + k_{x,y}(s_i, t_{j+1}) - k_{x,y}(s_i, t_j)\\[2pt]
&\hspace{10mm} + \frac{1}{2}\langle x_{s_{i+1}} - x_{s_i}, y_{t_{j+1}} - y_{t_j}\rangle\big(k_{x,y}(s_{i+1}, t_j) + k_{x,y}(s_i, t_{j+1})\big).
\end{align*}
It follows from theorem \ref{thm:goursat} that the PDE solution and boundary data are smooth on each small rectangle in $P_\lambda$. So by (\ref{eqn:error_estimate_2}) there exists $K_1\hspace{0.25mm} >\hspace{0.25mm} 0\hspace{0.25mm}$, depending on $M$ and $k_{x,y\,}$, such that
\begin{align*}
\big|k_{x,y}(s_{i+1},t_{j+1}) - k^{\prime}(s_{i+1},t_{j+1})\big| & \leq K_1(s_{i+1} - s_i)(t_{j+1} - t_j)\\[1pt]
&\hspace{12.5mm}\big((s_{i+1} - s_i)^2 + (s_{i+1} - s_i)(t_{j+1} - t_j) + (t_{j+1} - t_j)^2\big).
\end{align*} 
Taking the difference between $\tilde{k}(s_{i+1},t_{j+1})$ and $\hat{k}(s_{i+1}, t_{j+1})$ gives 
\begin{align*}
&\big|k^{\prime}(s_{i+1},t_{j+1}) - \hat{k}(s_{i+1}, t_{j+1})\big| \\
&\hspace{5mm}\leq \big|k_{x,y}(s_i, t_j) - \hat{k}(s_i, t_j)\big| + \frac{1}{2}\Big(1 + \big|\langle x_{s_{i+1}} - x_{s_i}, y_{t_{j+1}} - y_{t_j}\rangle\big|\Big)\big|k_{x,y}(s_{i+1}, t_j) - \hat{k}(s_{i+1}, t_j)\big|\\[2pt]
&\hspace{47mm} + \frac{1}{2}\Big(1 + \big|\langle x_{s_{i+1}} - x_{s_i}, y_{t_{j+1}} - y_{t_j}\rangle\big|\Big)\big|k_{x,y}(s_i, t_{j+1}) - \hat{k}(s_i, t_{j+1})\big|.
\end{align*}
Hence by the triangle inequality, we obtain a recurrence relation for the approximation errors,
\begin{align*}
&\big|k_{x,y}(s_{i+1},t_{j+1}) - \hat{k}(s_{i+1}, t_{j+1})\big|\\[2pt]
&\hspace{5mm} \leq \big|k_{x,y}(s_i, t_j) - \hat{k}(s_i, t_j)\big|\\[1pt]
&\hspace{10mm} + \frac{1}{2}\Big(1 + M(s_{i+1}-s_i)(t_{j+1}-t_j)\Big)\big|k_{x,y}(s_{i+1}, t_j) - \hat{k}(s_{i+1}, t_j)\big|\\[1pt]
&\hspace{10mm} + \frac{1}{2}\Big(1 + M(s_{i+1}-s_i)(t_{j+1}-t_j)\Big)\big|k_{x,y}(s_i, t_{j+1}) - \hat{k}(s_i, t_{j+1})\big|\\[2pt]
&\hspace{10mm} + K_1 (s_{i+1} - s_i)(t_{j+1} - t_j)\big((s_{i+1} - s_i)^2 + (s_{i+1} - s_i)(t_{j+1} - t_j) + (t_{j+1} - t_j)^2\big).
\end{align*}
Since each $(s_{i+1} - s_i)$ and $(t_{j+1} - t_j)$ is proportional to $2^{-\lambda}$, the result for the explicit scheme follows by iteratively applying the above recurrence relation. 
\end{proof}

\newpage

\section{Background on rough path theory}\label{appendix:RPT}

In this appendix we present a brief summary of rough path theory explain all the necessary meterial required to understand the results presented in \cref{sec:geom}, which should be self-contained. We begin by explaining what a \emph{controlled differential equation} (CDE) is. 

\subsection{Controlled differential equations (CDEs)}\label{appendix:CDE}

In a simplified setting where everything is differentiable a CDE is a differential equation of the form
\begin{equation}\label{eqn:CDE_intial}
    \frac{dy_t}{dt} = f\left(\frac{dx_t}{dt}, y_t\right), \hspace{0.5cm} y_0 = a
\end{equation}

where $x:[0,T] \to \mathbb{R}^d$ is a given path, $a \in \mathbb{R}^n$ is an initial condition, and $y:[0,T] \to \mathbb{R}^n$ is the unknown solution. If $f$ did not depend on its first variable, equation \cref{eqn:CDE_intial} would be a first order time-homogeneous ODE, the function $f$ would be a vector field, and $y$ would be the integral curve of $f$ starting at $a$. If instead we had $\frac{dy_t}{dt} = f(t,y_t)$, then \cref{eqn:CDE_intial} would describe a first order time-inhomogeneous ODE, with solution $y$ being an integral curve of a time-dependent vector field $f$. In the CDE \cref{eqn:CDE_intial} the time-inhomogeneity depends on the trajectory of $x$, which is said to control the problem. \medbreak

A CDE provides a fairly generic way to describe how a signal $x$ interacts with a control system of the form \cref{eqn:CDE_intial} to produce a response $y$. For example, a CDE can model the interaction of an electricity signal (current, voltage) with a domestic appliance (e.g. a washing machine) to produce a response (e.g. the rotation speed or the water temperature). \medbreak

Throughout this appendix $V,W$ will be two Banach spaces and $L(V,W)$ will denote the set of bounded linear maps from $V$ to $W$. We will denote by $I=[0,T]$ a generic compact (time) interval. Consider two continous paths $x:I \to V$ and $y:I \to W$ and a continuous function $f:W \to L(V,W)$, and let's assume that $x,y$ and $f$ are regular enough for the integral $\int_0^t f(y_s)dx_s$ to make sense for any $t \in I$\footnote{In what follows we will discuss various regularity assumptions on $f$.}. \medbreak 

Given an initial condition $a \in W$, we say that $x,y$ and $f$ satisfy the following controlled differential equation (CDE)
\begin{equation}\label{eqn:CDE}
    dy_t = f(y_t)dx_t, \hspace{0.5cm} y_0=a
\end{equation}
if the following equality holds for all $t \in I$
\begin{equation}\label{eqn:CDE_integral}
    y_t = a + \int_0^t f(y_s)dx_s 
\end{equation}
Here $f$ is called a \emph{vector field}, $x$ the \emph{control} driving the equation and $y$ the \emph{response} solution.

\begin{remark}
If $y$ is a solution of \cref{eqn:CDE} driven by $x$ and if $\psi :I \to I$ is an increasing surjection (also called time-reparametrization), then $y \circ \psi$ is also a solution to the same equation driven by $x \circ \psi$. In other words, solutions of CDEs are invariant to time-reparametrization.
\end{remark}

\begin{remark}\label{remark:vector_field}
The definition of $f$ as a continuous function from $W$ to $L(V,W)$ makes the notation $f(y_s)dx_s$ in \cref{eqn:CDE} clear. However, there is another equivalent interpretation of $f$ which is important to keep in mind. Let $C(W,W)$ be the space of continuous functions from $W$ to $W$. Then $f$ can be thought of as an element of $L(V,C(W,W))$. Adopting this point of view, $f$ is a linear function on $V$ with values in the space of vector fields on $W$. To make this view compatible with \cref{eqn:CDE} it is perhaps easier to think of the CDE \cref{eqn:CDE} using a slightly different notation: $dy_t = f(dx_t)y_t$. This alternative point of view has a deep physical interpretation: at each time $t$ the solution $y$ describes the state of a complex system that evolves as a function of its present state $y_t$ and of an infinitesimal external parameter $dx_t$ controlling it. The function $f$ transforms the infinitesimal displacement $dx_t$ into a vector field $f(dx_t)$ that defines a direction for the trajectory of $y$ to follow.
\end{remark}

\subsection{CDEs driven by paths of bounded variation}

\begin{definition}
A continuous path $x : I \to V$ is said to be of bounded variation if
\begin{equation}
    \sup_{\mathcal{D}} \sum_{t_i \in \mathcal{D}} |x_{t_{i+1}} - x_{t_i}| < + \infty
\end{equation}
where the supremum is taken over all partitions $\mathcal{D}$ of an interval $I$, i.e. over all increasing sequences of ordered (time) indices such that $\mathcal{D}=\{0=k_0 < k_1 <... < k_r=T\}$. We denote by $\Omega_1(V)$ the space of continuous paths of bounded variation with values in $V$.
\end{definition}


\begin{definition}\label{def:p-variation}
    Let $p \geq 1$ be a real number and $x: I \to V$ be a continuous path. The $p$-variation of $x$ on the interval $I$ is defined by
    \begin{equation}
        ||x||_{p,I} = \left( \sup_{\mathcal{D}} \sum_{t_i \in \mathcal{D}} |x_{t_{i+1}} - x_{t_i}|^p \right)^{1/p}
    \end{equation}
\end{definition}

\begin{remark}
    For any continuous path $x:I \to V$, any time-reparametrization $\psi : I \to I$ and scalar $p \geq 1$ one has $||x||_{p,I} = ||x \circ \psi||_{p,I}$; in other words the $p$-variation of $x$ is invariant to time-reparametrization. Furthermore, the function $p \to ||x||_{p,I}$ is decreasing. Hence, any path of finite $p$-variation is also a path of finite $q$-variation for any $q > p$. 
\end{remark}

If one assumes that $x$ is of bounded variation and that $y$ and $f$ are continuous, then the integral $\int_0^t f(y_s)dx_s$ exists for all $t \in I$ (as a classical Riemann-Stieltjes integral). If $f$ is Lipschitz continuous then the solution is unique \cite[Theorem 1.3 \& 1.4]{lyons2004differential}. 

\subsection{CDEs driven by paths of finite $p$-variation with $p < 2$}

However, If our goal is to solve differential equations driven by more irregular paths we should be prepared to compensate by allowing more smoothness on the function $f$. In particular,  if $p$ and $\gamma$ are real numbers such that $1 \leq p < 2$ and $p-1< \gamma\leq 1$ and $x$ has finite $p$-variation, and if $f$ is H{\"o}lder continuous with exponent $\gamma$, then the CDE \cref{eqn:CDE} admits a solution \cite[Theorem 1.20]{lyons2004differential}. In order to get uniqueness we require $f$ to be more regular than  H{\"o}lder continuous. In the next definition we introduce a class of functions exhibiting the appropriate level of regularity to ensure uniqueness.



\begin{definition}\label{def:Lip_form}
Let $V,W$ be Banach spaces, $k \geq 0$ an integer, $\gamma \in (k,k+1]$ and $C$ a closed subset of $V$. Let $f: C \to W$ be a function. For each integer $n=1,...,k$, let $f^n : C \to L(V^{\otimes n}, W)$ be a function taking its values in the space of symmetric $n$-linear forms from $V$ to $W$, where $\otimes$ denotes the tensor product of vector space (that we will define more precisely in the next section). The collection $(f_0=f,f^1,...,f^k)$ is an element of $Lip(\gamma,C)$ if there exists a constant $M \geq 0$ such that for each $n=0,...,k$
\begin{equation}
    \sup_{x \in C}|f^n(x)| \leq M
\end{equation}
and there exists a function $R_n : V \times V \to L(V^{\otimes n}, W)$ such that for any $x,y \in C$ and any $v \in V^{\otimes n}$
\begin{equation}
    f^n(y)(v) = \sum_{i=0}^{k-n} \frac{1}{i!}f^{n+i}(x)(v \otimes (y-x)^{\otimes i}) + R_{n}(x,y)(v)
\end{equation}
and
\begin{equation}
    |R_n(x,y)| \leq M |x-y|^{\gamma - n}
\end{equation}
The smallest constant $M$ for which the inequalities above hold for all $n$ is called the $Lip(\gamma,C)$-norm of $f$ and is denoted by $||f||_{Lip(\gamma)}$. We will call such function $f$ a $Lip(\gamma)$ function (or one-form).
\end{definition}

\begin{remark}
    The best way to understand of a $Lip(\gamma)$ function is to think about it as a function that ``locally looks like a polynomial function''. Let us explain the connection to paths. Let $k \geq 0$ be an integer. Let $P:V \to \mathbb{R}$ be a polynomial of degree $k$. Let $x:[0,T] \to V$ be a continuous path of bounded variation. If $P^1 : V \to L(V,\mathbb{R})$ denotes the derivative of $P$ then by Taylor's theorem we have that for any $0\leq s \leq t \leq T$
    \begin{equation}
        P(x_t) = P(x_s) + \int_s^tP^1(x_u)dx_u
    \end{equation}
    Substituting the expression of $P$ into $P^1$ in the last expression we get
    \begin{equation}
        P(x_t) = P(x_s) + P^1(x_s)\int_{s<u_1<t}dx_{u_1} + \underset{s<u_1<u_2<t}{\int\int}P^2(x_{u_1})dx_{u_1}\otimes dx_{u_2}
    \end{equation}
    where $P^2 : V \to L(V \otimes V, \mathbb{R})$ is the second derivative of $P$, and $L(V \otimes V, \mathbb{R})$ is the space of bilinear forms on $V$. Iterating the substitutions, the procedure stops at level $k+1$ yielding the following expression
    \begin{align}
        P(x_t) &= P(x_s) + P^1(x_s)\int_{s<u_1<t}dx_{u_1} + P^2(x_s)\underset{s<u_1<u_2<t}{\int\int}dx_{u_1}\otimes dx_{u_2} \nonumber \\
        &+ ... + P^k(x_s) \underset{s<u_1<...<u_k<t}{\int...\int}dx_{u_1}\otimes ... \otimes dx_{u_k} \label{eqn:poly_form_}
    \end{align}
    where for each $n \in \{0,...,k\}$, $P^n : V \to L(V^{\otimes n}, \mathbb{R})$ denotes the $n^{th}$ derivative of $P$ which takes its values in the space of symmetric $n$-linear forms on $V$. The symmetric part of the tensor $\underset{0<u_1<...<u_k<t}{\int...\int}dx_{u_1}\otimes ... \otimes dx_{u_k}$ is equal to $\frac{1}{k!}(x_t-x_s)^{\otimes k}$. Hence
    \begin{equation}\label{eqn:poly_form}
        P(x_t) = P(x_s) + \sum_{n=1}^kP^n(x_s)\frac{(x_t-x_s)^{\otimes k}}{k!}
    \end{equation}
    The general concept of $Lip(\gamma)$ function where $\gamma \in (k,k+1]$ for some $k \geq 0$ mimics closely the behaviour defined by equation \cref{eqn:poly_form} up to an error which is H{\"o}lder continuous of exponent $\gamma - k$. Indeed, using the notation from \cref{def:Lip_form}, if $x:[0,T] \to C$ takes its values in a closed subset $C$ of $V$, for any $0\leq s \leq t \leq T$, any $n=0,...,k$ and any $v \in V^{\otimes n}$ we have the following relation
    \begin{equation}
        f^n(x_t)(v) = \sum_{i=0}^{k-n} f^{n+i}(x_s)\left(v \otimes \underset{0<u_1<...<u_i<t}{\int... \int}dx_{u_1}\otimes...\otimes dx_{u_i}\right) + R_n(x_s,x_t)(v)
    \end{equation}
    which is analogous to \cref{eqn:poly_form_}. This is the general form of a $Lip(\gamma)$ function that we will consider from now on.
\end{remark}

If $p$ and $\gamma$ are such that $1\leq p<2$ and $p<\gamma$, $x:[0,T] \to V$ is a continuous path of finite $p$-variation and $f$ be a $Lip(\gamma)$ function, then for any $a \in W$, the CDE \cref{eqn:CDE} admits a unique solution. Furthermore, if $y=I_f(x,a)$ denotes the unique solution, then the function $I_f$, that maps $V$-valued paths of finite $p$-variation to $W$-valued paths of finite $p$-variation, is continuous in the $||\cdot||_p$ norm \cite[Theorem 1.28]{lyons2004differential}. \medbreak


We are now able to give meaning to CDEs where the control is not necessarily of bounded variation. However, if the driving path $x$ has finite $2$-variation but infinite $p$-variation for every $p<2$, we are still not able to give a meaning to the CDE \cref{eqn:CDE}. This is quite a hard restriction given that Brownian paths have finite $p$-variation only for $p>2$. \medbreak

To be able to push this barrier even further we first need to introduce some algebraic spaces and construct one of the central objects in rough path theory, the signature of a path. We do so very briefly in the next section and then explain in \cref{subsec:role_signature_CDEs} the connection with (linear) CDEs. 

\begin{remark}
    The $p$-variation $w(s,t) = ||x||_{p,[s,t]}^p$ function defines a control.
\end{remark}

\subsection{The role of the signature to solve linear CDEs}\label{subsec:role_signature_CDEs}

It turns out that the sequence of iterated integrals of a path (i.e. its signature) in \cref{def:signature} appears naturally when one wants to solve linear CDEs. In what follows, we show how the iterated integrals of a path precisely encode all the information that is necessary to determine the response of a linear CDE driven by this path. \medbreak



Let $V,W$ be two Banach spaces. Let $A : V \to L(W,W)$ be a bounded linear map. Here $A$ will play the role of (linear) vector field defined with the alternative but equivalent point of view expressed in \cref{remark:vector_field}. Let $x : [0,T] \to V$ be a continuous path of bounded variation. Consider the following system of linear CDEs
\begin{align}
    dy_t &= Ay_tdx_t, \hspace{0.5cm} y_0 \in W \label{eqn:CDE1}\\
    d\phi_t &= A\phi_t dx_t, \hspace{0.5cm} \phi_0 \in L(W,W) \label{eqn:CDE2}
\end{align}
where in the first equation the expression $Ay_tdx_t$ is to be read as $A(dx_t)(y_t)$ and in the second expression $A\phi_tdx_t$ means $A(dx_t) \circ \phi$. The solution $t \mapsto \phi_t$ to \cref{eqn:CDE2} is called the flow associated to the linear equation \cref{eqn:CDE1}. From the flow $\phi_t$ and the initial condition $y_0$ one gets the solution $y_t$ in the usual way
\begin{equation}
    y_t = \phi_t(y_0)
\end{equation}
The map $I_A : \Omega_1(V) \times W \to \Omega_1(W)$ that to the pair $(x,y_0)$ associates the solution $y$ is generally called the It\^o-map. We will now run an iterative procedure analogous to the Picard's iteration for ODEs in order to obtain a sequence $\{\phi^n_t : [0,T] \to L(W,W)\}_{n \geq 0}$ of flows, i.e. paths with values in the space of linear vector fields on $W$. Firstly, we start by setting the first element of the sequence $\phi^0_t = I_W$ equal to the identity on $L(W,W)$ and then we integrate \cref{eqn:CDE2} to get
\begin{equation}
    \phi^1_t = I_W + \int_0^t A\phi_s^0dx_s = I_W + \int_0^t A(dx_s)
\end{equation}
Doing it again we obtain the following expression for the next element in the sequence
\begin{equation}
    \phi^2_t = I_W + \int_0^t A\phi_s^1dx_s = I_W + \int_0^t A(dx_s) + \underset{0<s_1<s_2<t}{\int\int}A(dx_{s_1})A(dx_{s_2})
\end{equation}
Iterating this process we get the following expression for the $n^{th}$ term in the sequence
\begin{equation}\label{eqn:phi^n_t}
    \phi^n_t = I_W + \sum_{k=1}^n\underset{0<s_1<...<s_k<t}{\int...\int}A(dx_{s_1})... A(dx_{s_k})
\end{equation}
Now, for any $k \geq 1$ the linear operator $A : V \to L(W,W)$ can be extended to the operator $A^{\otimes k} : V^{\otimes k} \to L(W,W)$ as follows 
\begin{equation}
    A^{\otimes k}(v_1 \otimes ... \otimes v_k) = A(v_k)... A(v_1)
\end{equation}
which allows us to rewrite equation \cref{eqn:phi^n_t} as
\begin{equation}
    \phi^n_t = I_W + \sum_{k=1}^n A^{\otimes k}\underset{0<s_1<...<s_k<t}{\int...\int}dx_{s_1} \otimes ... \otimes dx_{s_k}
\end{equation}
where we can finally recognise the various terms of the signature of the path $x$ (see Definition \ref{def:signature}) and how the latter fundamentally encode the information to solve linear CDEs. The sequence of approximations $\{\phi^n\}$ for the flow $\phi$ provides a sequence of approximations for the solution of the CDE \cref{eqn:CDE1}:
\begin{equation}\label{eqn:signature_expansion_linear}
    y_t^n = y_0 + \sum_{k=1}^n \left(A^{\otimes k}\underset{0<s_1<...<s_k<t}{\int...\int}dx_{s_1} \otimes ... \otimes dx_{s_k}\right)y_0
\end{equation}
It turns out that the speed at which the sequence of approximations $\{y^n\}$ converges to the true solution $y$ of the CDE \cref{eqn:CDE1} is extremely high.

\begin{lemma}\cite[Proposition 2.2]{lyons2004differential}
For each $k \geq 1$ we have
\begin{equation}
    \Big|\underset{0<s_1<...<s_k<t}{\int...\int}dx_{s_1} \otimes ... \otimes dx_{s_k}\Big| \leq \frac{||x||^k_{1,[0,t]}}{k!}
\end{equation}
In particular, taking any norm $||\cdot||$ on bounded linear operators, the error decays factorially
\begin{equation}
    |y_t - y_t^n| \leq \sum_{k=n+1}^\infty \frac{(||A||\cdot||x||_{1,[0,t]})^k}{k!}
\end{equation}
\end{lemma}

This shows how the signature of a control $x$ provides an extremely efficient sequence of statistics to solve any linear CDE driven by $x$ (provided the norm of $A$ is not too large).

\subsection{Multiplicative functionals and the Extension Theorem}

In this section we push the limit on the regularity of the driver even further and explain how to solve CDEs driven by rough paths, as presented in the seminal paper \cite{lyons1998differential}. This will demand to define what a (geometric) rough path is and to present a powerful theory of integration of one-forms along rough paths to what will follow the Universal Limit Theorem for CDEs driven by rough paths (\cref{thm:ULT}). In what follows, $\Delta_T$ will denote the simplex
\begin{equation}
    \Delta_T = \{(s,t) \in [0,T]^2 : 0 \leq s \leq t \leq T\} 
\end{equation} \medbreak

\begin{definition}\label{def:multiplicative_functional}
Let $N \geq 1$ be an integer and let $X : \Delta_T \to T^N(V)$ be a continuous map
\begin{equation}
    X_{s,t} = (X^0_{s,t}, X^1_{s,t}, ..., X^N_{s,t}) \in \mathbb{R}\oplus V \oplus V^{\otimes 2} \oplus ... \oplus V^{\otimes N}
\end{equation}
$X$ is said to be a multiplicative functional if $X^0_{s,t}=1$ for all $(s,t) \in \Delta_T$ and
\begin{equation}
    X_{s,u}\otimes X_{u,t} = X_{s,t}, \hspace{0.5cm} \forall s,u,t \in [0,T], \ s \leq u \leq t
\end{equation} \medbreak
\end{definition}

The \emph{Chen's identity} condition of \cref{def:multiplicative_functional} imposed to a multiplicative functional is a purely algebraic one. We are now going to describe an analytic condition for a multiplicative functional through the notion of $p$-variation. \medbreak

\begin{definition}\label{def:control}
A control is a continuous non-negative function $\omega : \Delta_T \to [0, + \infty)$ which is super-additive in the sense that
\begin{equation}
    w(s,t) + w(t,u) \leq w(s,u), \ \ \ \forall s \leq t \leq u \in I
\end{equation}
and for which $w(t,t)=0$ for all $t \in I$.
\end{definition}

\begin{definition}
Let $p \geq 1$ be a real number and $N \geq 1$ be an integer. Let $\omega:[0,T] \to [0,+\infty)$ be a control and $X:\Delta_T \to T^N(V)$ be a multiplicative functional. We say that $X$ has finite $p$-variation on $\Delta_T$ controlled by $\omega$ if 
\begin{equation}\label{eqn:control_bound}
||X_{s,t}^i||_{V^{\otimes i}} \leq \frac{\omega(s,t)^{i/p}}{\beta(i/p)!}, \ \ \forall (s,t) \in \Delta_T, \ \forall i=1,...,N
\end{equation}
where $(i/p)! = \Gamma(i/p)$ and where $\beta$ is a real constant that depends only on $p$ and that we will make explicit later in the chapter. We say that $X$ has finite $p$-variation if there exists a control $\omega$ such that this condition is satisfied.
\end{definition} \medbreak

\begin{definition}
Let $X, Y : \Delta_T \to T^N(V)$ be two multiplicative functionals. The $p$-variation metric is defined as follows
\begin{equation}
d_p(X, Y) = \sup_{1 \leq i \leq N}\sup_{\mathcal{D}}\left(\sum_{t_k \in \mathcal{D}}||X^i_{t_k, t_{k+1}} - Y^i_{t_k, t_{k+1}}||^{p/i}\right)^{1/p}
\end{equation}
where $||\cdot||$ denotes any norm on the corresponding level $V^{\otimes i}$.
\end{definition} \medbreak

The next theorem is one of the fundamental results in rough path theory. It states that every multiplicative functional of degree $N$ and of finite $p$-variation can be extended in a unique way to a multiplicative functional of arbitrarily high degree provided $N$ is greater than the integer part of $p$, denoted by $\lfloor p \rfloor$. Furthermore the extension map is continuous in the $p$-variation metric.

\begin{theorem}\cite[Extension Theorems 3.7 \& 3.10]{lyons2004differential}\label{thm:extension_theorem}
Let $p \geq 1$ be a real number and $N \geq 1$ an integer. Let $X : \Delta_T \to T^N(V)$ be a multiplicative functional of finite $p$-variation controlled by a control $\omega$. Assume that $N \geq \lfloor p \rfloor$. Then, for any $m \geq \lfloor p \rfloor + 1$, there exist a unique continuous function $X^m: \Delta_T \to V^{\otimes m}$ such that
\begin{equation}
    (s,t) \mapsto X_{s,t} = (1, X_{s, t}^1, ... , X_{s, t}^{\lfloor p \rfloor}, ... , X_{s, t}^m , ... ) \in T((V))
\end{equation}
is a multiplicative functional of finite $p$-variation controlled by $w$ according to the inequality \cref{eqn:control_bound} with
\begin{equation}
    \beta = p^2\left(1 + \sum_{r=3}^\infty \left(\frac{2}{r-2}\right)^{\frac{\lfloor p \rfloor + 1}{l}}\right)
\end{equation}
Moreover, the extension map is continuous in the $p$-variation metric in the sense that if for some $\epsilon > 0$ we have
\begin{equation}\label{eqn:ee}
    ||X^i_{s,t} - Y^i_{s,t}|| \leq \epsilon \frac{\omega(s,t)^{i/p}}{\beta(i/p)!}, \hspace{0.5cm} i=1,...,N, \ (s,t) \in \Delta_T
\end{equation}
then, provided 
\begin{equation}
    \beta \geq  2p^2\left(1 + \sum_{r=3}^\infty \left(\frac{2}{r-2}\right)^{\frac{\lfloor p \rfloor + 1}{l}}\right)
\end{equation}
the bound \cref{eqn:ee} holds for all $i \geq 1$.
\end{theorem}

\begin{definition}
Let $V$ be a Banach space and $p \geq 1$ be a real number. A $p$-rough path in $V$ is a multiplicative functional of degree $\lfloor p \rfloor$ with finite $p$-variation. We denote the space of $p$-rough paths over $V$ by $\Omega_p(V)$.
\end{definition}

Hence, a $p$-rough path is a continuous map from $\Delta_T \to T^{ \lfloor p \rfloor}(V)$ that satisfies an algebraic condition (Chen's identity) and an analytic condition (finite $p$-variation). As a consequence of the Extension Theorem (\ref{thm:extension_theorem}), we get that every $p$-rough path in $\Omega_p(V)$ has a full signature, i.e. it can be extended to a multiplicative functional of arbitrary high degree with finite $p$-variation in $V$.

\begin{remark}
The signature of a paths of bounded variation, truncated at any level, is a multiplicative functional, so in particular it satisfies Chen's identity. This allows for fast signature computations for time series data available in several software packages \cite{esig, signatory, reizenstein2018iisignature}.    
\end{remark}

Recall that a path of finite $p$-variation has finite $q$-variation for any $q \geq p$. In particular, a path of bounded variation is a $p$-rough path for any $p \geq 1$. We now dispose of all the necessary ingredients to define what a geometric $p$-rough path is.

\begin{definition}\label{def:geom_rough_path}
A geometric $p$-rough path is a $p$-rough path that can be expressed as the limit of a sequence of $1$-rough paths in the $p$-variation metric. We denote by $G \Omega_p (V)$ the space of geometric $p$-rough paths in $V$.
\end{definition}

Hence, $G\Omega_p(V)$ is the closure in $(\Omega_p(V),d_p)$ of the space $\Omega_1(V)$ of continuous paths of bounded variation. Next we apply the ideas developed so far in order to define a powerful theory of integration for solving differential equations driven by (geometric) rough paths. 

\subsection{Integration of a one-form along a rough path}\label{sec:integration_one_form}

In order to define what we mean by a solution of a differential equation driven by a rough path, we need a theory of integration for rough paths. The correct notion of integral turns out to be the integral of a one-form along a rough path. We start by defining the concept of almost $p$-rough path.

\begin{definition}
Let $p \geq 1$ be a real number. Let $\omega : \Delta_I \to [0, + \infty)$ be a control. A function $X : \Delta_I \to T^{ \lfloor p \rfloor}(V)$ is called an almost $p$-rough path if 
\begin{enumerate}
    \item it has finite $p$-variation controlled by $w$, i.e.
    \begin{equation}
        \norm{X^i_{s,t}} \leq \frac{\omega(s,t)^{i/p}}{\beta(i/p)!}, \hspace{0.5cm} \forall (s,t) \in \Delta_t, \ \forall i=0,...,\lfloor p \rfloor
    \end{equation}
    \item it is an almost multiplicative functional, i.e. there exists $\theta > 1$ such that
    \begin{equation}
        ||(X_{s,u} \otimes X_{u,t})^i - X_{s,t}^i|| \leq \omega(s,t)^\theta, \ \ \forall (s,t) \in \Delta_I, \forall i=0,...,\lfloor p \rfloor
    \end{equation}
\end{enumerate}
To be specific we say that $X$ is a $\theta$-almost $p$-rough path controlled by $\omega$.
\end{definition}

For any almost $p$-rough path there exists a unique $p$-rough path that it is close to it in some specific sense, as stated in the next theorem.  

\begin{theorem}\cite[theorem 4.3 \& 4.4]{lyons2004differential}\label{thm:unique_rough_path}
Consider two scalars $p\geq 1$ and $\theta \geq 1$,  a control $\omega : \Delta_T \to [0,+\infty)$. Let $X: \Delta_T \to T^{ \lfloor p \rfloor}(V)$ be a $\theta$-almost $p$-rough path controlled by $\omega$. Then there exists a unique $p$-rough path $\widetilde X: \Delta_T \to  T^{ \leq \lfloor p \rfloor}(V)$ such that
\begin{equation}
\sup_{\substack{0\leq s< t \leq T \\  i=0,...,\lfloor p \rfloor}} \frac{||\widetilde X^i_{s,t} - X^i_{s,t}||}{\omega(s,t)^\theta} < +\infty
\end{equation}
Furthermore, the map that to an almost $p$-rough path associates a rough path is continuous in $p$-variation.
\end{theorem}

\begin{remark}
    It is important to note that, if $X$ and $Y$ are two $p$-rough paths with $p \geq 2$, and $f$ is a map, even the smoothest one, it is not possible to give a sensible meaning to something like $\int f(Y)dX$. In effect, as shown by the simple example $\int_s^tY_udX_u = \underset{s<u_1<u_2<t}{\int \int}dY_{u_1}dX_{u_2}$, the definition of such object would necessarily involve some cross-iterated integrals of $X$ and $Y$, which are not available directly in the data of $X$ and $Y$. In other words, two rough paths $X \in \Omega_p(V)$ and $Y \in \Omega_p(W)$ do not determine the joint path $(X,Y)$ as a rough path in $\Omega_p(V\oplus W)$. However, if one knows the joint path $Z=(X,Y)$ as a rough path, then the integral $\int f(Y)dX$ can be defined as a special case of $\int \alpha(Z)dZ$, where $\alpha$ is a one-form. This is the object we are going to construct in the sequel.
\end{remark}

We will define the integral of a one-form along a rough path as the unique $p$-rough path associated to, in the sense of  \cref{thm:unique_rough_path}, a specific almost $p$-rough path that we now construct. \medbreak

Let $\gamma>p\geq 1$ be scalars and let $Z : \Delta_T \to T^{ \lfloor p \rfloor}(V)$ be a $p$-rough path controlled by some control $\omega$. Let $\alpha : V \to L(V,W)$ be a $Lip(\gamma-1)$ function (\cref{def:Lip_form}) that we will call a one-form from now on. Hence, we are given $\alpha^0=\alpha$ and at least $\lfloor p \rfloor - 1$ auxiliary functions $\alpha^1, ..., \alpha^{\lfloor p \rfloor - 1}$ such that for any $k=1,...,\lfloor p \rfloor-1$
\begin{equation}
\alpha^k : V \to L(V^{\otimes k}, L(V, W)), \hspace{0.5cm} k=1 ... \lfloor p \rfloor - 1
\end{equation}
satisfying the Taylor-like expansion: $\forall x,y \in V$
\begin{equation}\label{eqn}
\alpha(y) =\alpha(x) + \sum_{k=1}^{\lfloor p \rfloor - 1} \alpha^k(x)\frac{(y-x)^{\otimes k}}{k!} + R_0(x,y)
\end{equation}
with $||R_0(x,y)|| \leq ||\alpha||_{Lip}||x-y||^{\gamma - 1}$, where $||\alpha||_{Lip}$ is the Lipschitz constant of $\alpha$. We are going to look for an approximation of $\int_s^t \alpha(Z_u)dZ_u$ in the form of a Taylor expansion. For any $(s,u) \in \Delta_T$ we can write
\begin{equation}\label{eqn2}
\alpha(Z_u) = \sum_{k=0}^{\lfloor p \rfloor -1}\alpha^k(Z_s)Z^k_{s,u} + R_0(Z_s, Z_u)
\end{equation}
Now, by definition of the iterated integrals of $Z$, the following relation holds for any $k \geq 0$ 
\begin{equation}\label{eqn3}
\int_s^t Z^k_{s,u} dZ_u  = Z^{k+1}_{s,t}
\end{equation}
where with the multiplication $Z^k_{s,u} dZ_u$ we mean $Z^k_{s,u} \otimes dZ_u$. Combining \cref{eqn2} and \cref{eqn3} we obtain
\begin{equation}
\int_s^t\alpha(Z_u) dZ_u = \sum_{k=0}^{\lfloor p \rfloor - 1}\alpha^k(Z_s)Z^{k+1}_{s,t} + \int_s^tR_0(Z_s, Z_u)dZ_u
\end{equation}

Let us now define the following $W$-valued path 
\begin{equation}
Y_{s,t}^1 := \sum_{k=0}^{\lfloor p \rfloor - 1}\alpha^k(Z_s)Z^{k+1}_{s,t}
\end{equation}
and let us compute the higher order iterated integrals of $Y$ in such a way that it becomes an almost rough path. For any $n \geq 2$ it can be shown after some tedious calculations that
\begin{equation}\label{eqn:Y_integral_one_form}
Y^n_{s,t} = \sum_{k_1,...,k_n=1}^{\lfloor p \rfloor} \alpha^{k_1-1}(Z_s) ... \alpha^{k_n-1}(Z_s) \underset{s<u_1<...<u_n<t}{\int ... \int}dZ_{s,u_1}^{k_1} \otimes ... \otimes dZ_{s,u_n}^{k_n}
\end{equation}
It turns out that the $Y^n$ we have just defined is an almost rough path, as stated in the next theorem. We will use this as our approximation for $\int \alpha(Z)dZ$.

\begin{theorem}\cite[Theorem 4.6]{lyons2007extension}\label{thm:almost_rough_path_Y}
Let $Z : \Delta_T \to T^{ \lfloor p \rfloor}(V)$ be a geometric $p$-rough path and $\alpha : V \to L(V,W)$ be a $Lip(\gamma-1)$ one-form for some $\gamma > p$. Then $Y:\Delta_T \to T^{\lfloor p \rfloor}(W)$ defined for all $(s,t) \in \Delta_T$ and any $n \geq 1$ by equation \cref{eqn:Y_integral_one_form} is a $\frac{\gamma}{p}$-almost $p$-rough path.
\end{theorem}

We can now define the integral $\int \alpha(Z)dZ$ of the one-form $\alpha$ along the rough path $Z$ as the unique rough path associated to the almost rough path $Y$ according to \cref{thm:unique_rough_path}.

\begin{definition}\label{def:rough_integral}
Let $Z, \alpha$ and $Y$ be as in \cref{thm:almost_rough_path_Y}. The unique $p$-rough path $\mathcal{I} : \Delta_T \to T^{ \lfloor p \rfloor}(W)$ associated to $Y$ by \cref{thm:unique_rough_path} is called the integral of the one-form $\alpha$ along the rough path $Z$ and is denoted by $\mathcal{I}_{s,t} := \int_s^t \alpha(Z_u)dZ_u$. 
\end{definition}

\begin{remark}
    As stated in \cite[Theorem 4.12]{lyons2004differential}, the map $Z \mapsto \int \alpha(Z)dZ$ is continuous in $p$-variation from $G\Omega_p(V)$ to $G\Omega_p(W)$.
\end{remark}

\begin{definition}\label{def:linear_image_rough_path}
    Let $A \in L(V,W)$ be a continuous linear map between Banach spaces and $X : \Delta_T \to T^{ \lfloor p \rfloor}(V)$ be a geometric $p$-rough path. The image $A(X)$ of the rough path $X$ by the linear map $A$ is a rough path defined in the following way. $A$ induces a linear map from $V^{\otimes k}$ to $W^{\otimes k}$ which sends $x_1\otimes ... \otimes x_k$ to $Ax_1 \otimes ... \otimes Ax_k$. By linearity of the direct sum defining $T^{ \lfloor p \rfloor}(V)$ and $T^{ \lfloor p \rfloor}(W)$, $A$ can be further extended to a linear map between these two spaces, denoted by $\mathcal{T}(A) \in L\left(T^{ \lfloor p \rfloor}(V), T^{ \lfloor p \rfloor}(W)\right)$. Then, $A(X)$ is defined as the following rough path
\begin{equation}
    A(X)_{s,t} = \mathcal{T}(A)(X_{s,t}), \hspace{0.5cm} \forall (s,t) \in \Delta_T
\end{equation}
\end{definition}

In the next section we will finally describe how to solve CDEs driven by rough paths. 

\subsection{Non-linear CDEs driven by rough paths}\label{sec:non_linear_CDEs}

In this section we explain how to solve CDEs driven by rough paths. As in the classical case of CDEs driven by paths of bounded variation, the existence and uniqueness of the solution depend on assumptions about the smoothness of the vector field.  \medbreak

Let $V,W$ be two Banach spaces and $\gamma>p \geq 1$ be two real numbers. Let $f : W \to L(V,W)$ be a $Lip(\gamma-1)$ one-form. Consider a geometric $p$-rough path $X \in G\Omega_p(V)$ and a point $a \in W$. It is easy to see that when $X$ has finite $1$-variation, the equation
\begin{equation}\label{eqn:CDE_rough_path}
    dY_t = f(Y_t)dX_t, \hspace{0.5cm} Y_0 = a
\end{equation}
is equivalent, up to translation of the initial condition $a \in W$, to the following system
\begin{align}
    dY_t &= f_a(Y_t)dX_t, \hspace{0.5cm} Y_0 = 0 \\
    dX_t &= dX_t
\end{align}
where $f_a(w) = f(w + a)$ for all $w \in W$. Consider now the one-form $h : V \oplus W \to L(V \oplus W, V \oplus W)$ defined as follows 
\begin{equation}\label{eqn:final_one_form}
    h(v,w) = \begin{pmatrix}
                I_V & 0 \\
                f_a(w) & 0 
             \end{pmatrix}
\end{equation}
where $I_V : V \to V$ is the identity map on $V$. Then, when $X \in \Omega_1(V)$ is of bounded variation, solving (\ref{eqn:final_one_form}) is equivalent to finding $Z \in \Omega_1(V \oplus W)$ such that
\begin{equation}
    dZ_t = h(Z_t)dZ_t, \hspace{0.5cm} Z_0 = 0, \ \pi_V(Z) = X
\end{equation}
where $\pi_V$ is the canonical projection of $V\oplus W$ onto $V$. What the next definition states is that this reformulation makes sense even when $X$ is a rough path.

\begin{definition}\label{def:solution_CDE_rough_path}
Consider the notation introduced above. We call $Z \in G\Omega_p(V \oplus W)$ a solution of the differential \cref{eqn:CDE_rough_path}
if the following two conditions hold:
\begin{align}
    &Z_{s,t} = \int_s^t h(Z_u)dZ_u\\
    &\pi_V(Z) = X
\end{align}
where $\pi_V(Z)$ is the image of the rough path $Z$ by the linear map $\pi_V$ in the sense of \cref{def:linear_image_rough_path}. 
\end{definition}

In order to find a solution to the CDE \cref{eqn:CDE_rough_path} we are going to, once again, use Picard iteration. Define $Z(0) = (X,\mathbf{0}) \in G\Omega_p(V \oplus W)$, where $\mathbf{0}$ denotes the $0$-rough path $\mathbf{0}_{s,t} = (1,0,0,...)$. Then, for any $n\geq 0$ we define the sequence of rough paths
\begin{equation}
    Z(n)_{s,t} = \int_s^t h(Z(n))dZ(n)
\end{equation}
Let us now denote $Y(n) = \pi_W(Z(n))$ so that $Z(n) = (X,Y(n))$. One of the main results in the seminal paper \cite{lyons1998differential} is the following Universal Limit Theorem (ULT). Note that for the solution to be unique we require one extra degree of smoothness on $f$.

\begin{theorem}[\textbf{Universal Limit Theorem (ULT)}]\cite[Theorem 5.3]{lyons2004differential}\label{thm:ULT}
Let $p \geq 1$ and $\gamma>p$ be real numbers. Let $f: W \to L(V,W)$ be a $Lip(\gamma)$ function. For all $X \in G\Omega_p(V)$ and all $a \in W$ the equation
\begin{equation}
    dY_t = f(Y_t)dX_t, \hspace{0.5cm} Y_0 = a
\end{equation}
admits a unique solution $Z=(X,Y) \in G\Omega_p(V \oplus W)$ in the sense of \cref{def:solution_CDE_rough_path}. Furthermore, the rough path $Y$ is the limit of of the sequence of rough paths $Y(n)$ defined above, and the mapping $I_f : G\Omega_p(V) \times W \to G\Omega_p(W)$ which sends $(X,a)$ to $Y$ is continuous in $p$-variation.
\end{theorem}

\section{Cross-integrals of the signature kernel}

In this final section of the appendix, we provide another interpretation for the double integral of \cref{def:double_integral_I}. We begin by recalling a generalization of the Extension Theorem \ref{thm:extension_theorem}.

\begin{definition}
Let $N \in \mathbb{N}$ be an integer. We denote by $G^N(V) \subset T^N(V)$ the free nilpotent Lie group of step $N$ over $V$.
\end{definition}

\begin{theorem}\cite[Theorem 14]{lyons2007extension}\label{thm:extension}
    Let $p>1$. Let $K$ be a closed normal subgroup of $G^{\lfloor p \rfloor}(V)$. If $x$ is a $(G^{\lfloor p \rfloor}(V)/ K)$-valued continuous path of finite $p$-variation, with $p \not\in \mathbb{N}\setminus \{0,1\}$, then there exists a continuous $G^{\lfloor p \rfloor}(V)$-valued geometric $p$-rough path $y$ such that
    \begin{equation*}
        \pi_{G^{\lfloor p \rfloor}(V), G^{\lfloor p \rfloor}(V)/ K}(y) = x
    \end{equation*}
    where $\pi_{G^{\lfloor p \rfloor}(V), G^{\lfloor p \rfloor}(V)/ K}$ is the canonical homomorphism (projection) from $G^{\lfloor p \rfloor}(V)$ to \\ $G^{\lfloor p \rfloor}(V) / K$.
\end{theorem}

\begin{corollary}
    If $p \in \mathbb{R}_{\geq 1} \setminus \{2, 3, ...\}$, then a continuous $V$-valued smooth path of finite $p$-variation can be lifted to a geometric $p$-rough path
\end{corollary}

\begin{proof}
    It suffices to apply Theorem \ref{thm:extension} to $K = \exp\{\bigoplus_{i=2}^{\lfloor p \rfloor} V_i\}$, where $v_0 = V$ and $V_{i+1} = [V, V_{i}]$, with $[\cdot, \cdot]$ being the Lie bracket.
\end{proof}

Without loss of generality let's assume $q \geq p$. $\mathbb{X}$ is a geometric $p$-rough path, therefore by the Extension Theorem $\mathbb{X}$ can be lifted uniquely to a geometric $q$-rough path $\mathbb{X}'$. Let $R$ be any compact time interval such that such that there exists two continuous and increasing surjections $\psi_1 : R \to I$ and $\psi_2 : R \to J$. Let $\widetilde{\mathbb{X}} = \mathbb{X}' \circ \psi_1$ and $\widetilde{\mathbb{Y}} = \mathbb{Y} \circ \psi_2$. \medbreak

Consider the path $\mathbb{Z}: R \to G^{\lfloor q \rfloor}(V) \times G^{\lfloor q \rfloor}(V)$ defined as
\begin{equation}
    \mathbb{Z}: t \mapsto (\widetilde{\mathbb{X}}_t, \widetilde{\mathbb{Y}}_t)
\end{equation}
    
$\mathbb{Z}$ is a continuous, $(G^{\lfloor q \rfloor}(V) \times G^{\lfloor q \rfloor}(V))$-valued path of finite $q$-variation. Firstly, we consider the \textit{product of algebras} $T^{\lfloor q \rfloor}(V) \times T^{\lfloor q \rfloor}(V)$, where the product of elements is defined by the following operation: $(f_1, g_1)(f_2, g_2) = (f_1 \otimes f_2, g_1 \otimes g_2)$. \medbreak

Now consider the free tensor algebra $T^{\lfloor q \rfloor}(V \bigoplus V)$ over the vector space $V \bigoplus V$. Let $\phi: V \to T^{\lfloor q \rfloor}(V)$ be the canonical inclusion of $V$ into $T^{\lfloor q \rfloor}(V)$ and let $\psi : T^{\lfloor q \rfloor}(V) \to T^{\lfloor q \rfloor}(V) \times T^{\lfloor q \rfloor}(V)$ be the linear map defined as $\psi(T) = (T, T), \forall T \in T(V)$. \medbreak

Now let's consider the map $\eta = \psi \circ \phi : V \to T^{\lfloor q \rfloor}(V) \times T^{\lfloor q \rfloor}(V)$. By the universal property of $\bigoplus$ there exists a unique algebra homomorphism $\Phi : V \bigoplus V \to T^{\lfloor q \rfloor}(V) \times T^{\lfloor q \rfloor}(V)$ such that $\Phi \circ \psi = \eta$. 
\[
\begin{tikzcd}[row sep=2.5em]
 & V \bigoplus V \arrow{dr}{\Phi} \\
V \arrow{ur}{\psi} \arrow{rr}{\eta} && T^{\lfloor q \rfloor}(V) \times T^{\lfloor q \rfloor}(V)
\end{tikzcd}
\]

But now $T^{\lfloor q \rfloor}(V \bigoplus V)$ has also the universal property, therefore there exists a unique algebra homomorphism $\Psi : T^{\lfloor q \rfloor}(V \bigoplus V) \to T^{\lfloor q \rfloor}(V) \times T^{\lfloor q \rfloor}(V)$ such that $\Psi \circ \beta = \Phi$, where $\beta$ is the canonical inclusion of $V \bigoplus V$ into $T^{\lfloor q \rfloor}(V \bigoplus V)$. 
\[
\begin{tikzcd}[row sep=2.5em]
 & T^{\lfloor q \rfloor}(V \bigoplus V) \arrow{dr}{\Psi} \\
V \bigoplus V \arrow{ur}{\beta} \arrow{rr}{\Phi} && T^{\lfloor q \rfloor}(V) \times T^{\lfloor q \rfloor}(V)
\end{tikzcd}
\]

Note that $G^{\lfloor q \rfloor}(V) \times G^{\lfloor q \rfloor}(V)$ is a group embedded in the product algebra $T^{\lfloor q \rfloor}(V) \times T^{\lfloor q \rfloor}(V)$ and $G^{\lfloor q \rfloor}(V\bigoplus V)$ is a group embedded in the tensor algebra $T^{\lfloor q \rfloor}(V \bigoplus V)$. Let $\pi$ be the map $\Psi$ restricted to $G^{\lfloor q \rfloor}(V \bigoplus V)$. Given that $G^{\lfloor q \rfloor}(V) \times G^{\lfloor q \rfloor}(V) \subset \pi(G^{\lfloor q \rfloor}(V \oplus V))$, this map is a surjective group-homomorphism. Therefore, by the \textit{First Group Isomorphism Theorem} we have that $Ker(\pi) \triangleleft G^{\lfloor q \rfloor}(V \oplus V)$, and 
    \begin{equation*}
        G^{\lfloor q \rfloor}(V \oplus V) / Ker(\pi) \simeq G^{\lfloor q \rfloor}(V) \times G^{\lfloor q \rfloor}(V)
    \end{equation*}
    
By Lemma \ref{thm:extension} there exists a continuous $G^{\lfloor q \rfloor}(V \oplus V)$-valued geometric $q$-rough path $\widetilde{\mathbb{Z}}$ such that $\pi(\widetilde{\mathbb{Z}}) = \mathbb{Z}$. Expanding out coordinate-wise we obtain 
    
    {\scriptsize
    \begin{align}
        \int_{s=u}^{u'} \int_{t=v}^{v'} k(\mathbb{X}_{u, s},\mathbb{Y}_{v, t}) \langle d\mathbb{X}_s, d\mathbb{Y}_t \rangle &= \sum_{n=0}^{\lfloor q \rfloor}\sum_{\tau \in \{1,..., d\}^n} \int_{s=u}^{u'} \int_{t=v}^{v'} k(\mathbb{X}_s,\mathbb{Y}_t) d\mathbb{X}_s^\tau d\mathbb{Y}_t^\tau \nonumber \\
        &= \sum_{n=0}^{\lfloor q \rfloor}\sum_{\tau \in \{1,..., d\}^n} \int_{s=u}^{u'} \int_{t=v}^{v'} \langle S(\mathbb{X}_{u, s}), S(\mathbb{Y}_{v, t}) \rangle d\mathbb{X}_s^\tau d\mathbb{Y}_t^\tau \nonumber \\
        &= \sum_{m=0}^\infty \sum_{\omega \in \{1,..., d\}^m} \sum_{n=0}^{\lfloor q \rfloor}\sum_{\tau \in \{1,..., d\}^n} \int_{s=u}^{u'} \int_{t=v}^{v'} S(\mathbb{X}_{u, s})^\omega S(\mathbb{Y}_{v, t})^\omega d\mathbb{X}_s^\tau d\mathbb{Y}_t^\tau
        \nonumber \\
 &= \sum_{m=0}^\infty \sum_{\omega \in \{1,..., d\}^m} \sum_{n=0}^{\lfloor q \rfloor}\sum_{\tau \in \{1,..., d\}^n}  \Big( \int_{s=u}^{u'} S(\mathbb{X}_{u, s})^\omega d\mathbb{X}_s^\tau \Big) \Big( \int_{t=v}^{v'} S(\mathbb{Y}_{v, t})^\omega d\mathbb{Y}_t^\tau \Big) \label{eqn:fact}
    \end{align}
    }

Note that all the cross-integrals of $S(\mathbb{X})$ and $S(\mathbb{Y})$ do not contribute in the above expression, which nicely factors into two separate integrals: expression (\ref{eqn:fact}) tells us that the rough path $\widetilde{\mathbb{Z}}$ does not depend on the lift used in the extension (from the joint path $\mathbb{Z}$ to the rough path $\widetilde{\mathbb{Z}}$). The terms involved in the infinite sum on the right-hand-side of the equation (\ref{eqn:fact}) are all $\mathbb{R}$-projections of the signature of the signautre of the rough paths $\mathbb{X}$ and $\mathbb{Y}$.

\end{document}